\documentclass[12pt,leqno,a4paper]{article}
\usepackage{amsthm}
\usepackage{indentfirst}
\usepackage{amsmath}
\usepackage{mathrsfs}
\usepackage{txfonts}
\usepackage{paralist}
\usepackage{amssymb,enumerate}
\usepackage{hyperref}
\usepackage{titlesec}
\usepackage{url}

\newtheorem{mainthm}{Theorem}
\newtheorem*{conjecture}{Conjecture}
\newtheorem{proposition}{Proposition}[section]

\newtheorem{lemma}[proposition]{Lemma}
\newtheorem{theorem}[proposition]{Theorem}
\newtheorem{remark}[proposition]{Remark}

\newtheorem{corollary}[proposition]{Corollary}
\newtheorem{notation}[proposition]{Notation}

\newcommand{\GL}{\operatorname{GL}\nolimits}
\newcommand{\GU}{\operatorname{GU}\nolimits}
\newcommand{\PSL}{\operatorname{PSL}\nolimits}
\newcommand{\CSp}{\operatorname{CSp}\nolimits}
\newcommand{\Sp}{\operatorname{Sp}\nolimits}

\newcommand{\PSp}{\operatorname{PSp}\nolimits}
\newcommand{\GO}{\operatorname{GO}\nolimits}
\newcommand{\SO}{\operatorname{SO}\nolimits}
\newcommand{\D}{\operatorname{D}\nolimits}
\newcommand{\Ker}{\operatorname{Ker}\nolimits}
\newcommand{\bl}{\operatorname{bl}\nolimits}
\newcommand{\Irr}{\operatorname{Irr}\nolimits}
\newcommand{\IBr}{\operatorname{IBr}\nolimits}
\newcommand{\Ind}{\operatorname{Ind}\nolimits}
\newcommand{\Res}{\operatorname{Res}\nolimits}
\newcommand{\Bl}{\operatorname{Bl}\nolimits}
\newcommand{\Aut}{\operatorname{Aut}\nolimits}

\newcommand{\Out}{\operatorname{Out}\nolimits}
\newcommand{\diag}{\operatorname{diag}\nolimits}

\newcommand{\dz}{\operatorname{dz}\nolimits}
\newcommand{\St}{\operatorname{St}\nolimits}
\newcommand{\rO}{\operatorname{O}\nolimits}
\newcommand{\I}{\operatorname{I}\nolimits}
\newcommand{\J}{\operatorname{J}\nolimits}
\newcommand{\rk}{\operatorname{rk}\nolimits}
\newcommand{\Jor}{\operatorname{Jor}\nolimits}

\newcommand{\zero}{\mathbf{0}}
\newcommand{\F}{\mathbb{F}}
\newcommand{\barF}{\overline{\mathbb{F}}}
\newcommand{\Z}{\mathbb{Z}}

\newcommand{\cF}{\mathcal{F}}

\newcommand{\hF}{\hat{F}}

\newcommand{\bV}{\mathbf{V}}
\newcommand{\bW}{\mathbf{W}}
\newcommand{\bG}{\mathbf{G}}
\newcommand{\bL}{\mathbf{L}}
\newcommand{\tG}{\widetilde{G}}
\newcommand{\tbG}{\widetilde{\mathbf{G}}}
\newcommand{\tbL}{\widetilde{\mathbf{L}}}

\newcommand{\tC}{\widetilde{C}}
\newcommand{\tN}{\widetilde{N}}
\newcommand{\hN}{\hat{N}}
\newcommand{\tB}{\widetilde{B}}
\newcommand{\tR}{\widetilde{R}}
\newcommand{\fS}{\mathfrak{S}}
\newcommand{\tQ}{\widetilde{Q}}

\newcommand{\cW}{\mathcal{W}}
\newcommand{\cE}{\mathcal{E}}
\newcommand{\cB}{\mathcal{B}}
\newcommand{\bc}{\mathbf{c}}

\newcommand{\cR}{\mathcal{R}}
\newcommand{\sC}{\mathscr{C}}

\newcommand{\cZ}{\mathcal{Z}}
\newcommand{\cJ}{\mathcal{J}}
\newcommand{\cK}{\mathcal{K}}

\newcommand{\hz}{\hat{z}}
\newcommand{\ts}{\tilde{s}}
\newcommand{\tb}{\tilde{b}}

\newcommand{\htheta}{\hat{\theta}}
\newcommand{\ttheta}{\tilde{\theta}}
\newcommand{\tsigma}{\tilde{\sigma}}
\newcommand{\tDelta}{\tilde{\Delta}}

\newcommand{\tchi}{\tilde{\chi}}

\newcommand{\homega}{\hat{\omega}}

\newcommand{\udots}{\mathinner{\mskip1mu\raise1pt\vbox{\kern7pt\hbox{.}}  
\mskip2mu\raise4pt\hbox{.}\mskip2mu\raise7pt\hbox{.}\mskip1mu}} 

\begin{document}

\title{An equivariant bijection between irreducible Brauer characters and weights for $\Sp(2n,q)$
\footnote{The author gratefully acknowledges financial support by NSFC (No. 11631001).
\newline \textbf{2010 Mathematics Subject Classification:} 20C20, 20C33.
\newline \textbf{Keyword:} Alperin weight conjecture, inductive condition, equivariant bijections, symplectic groups.}}

\author{Conghui Li\\
Department of Mathematics, Southwest Jiaotong University,\\
Chengdu 611756, China\\
Email:liconghui@swjtu.edu.cn}

\maketitle

\begin{abstract}
The longstanding Alperin weight conjecture and its blockwise version have been reduced to simple groups recently by Navarro, Tiep, Sp\"ath and Koshitani.
Thus, to prove this conjecture, it suffices to verify the corresponding inductive condition for all finite simple groups.
The first is to establish an equivariant bijection between irreducible Brauer characters and weights for the universal covering groups of simple groups.
Assume $q$ is a power of some odd prime $p$.
We first prove the blockwise Alperin weight conjecture for $\Sp_{2n}(q)$ and odd non-defining characteristics.
If the decomposition matrix of $\Sp_{2n}(q)$ is unitriangular with respect to an $\Aut(\Sp_{2n}(q))$-stable basic set (this assumption holds for linear primes), we can establish an equivariant bijection between the irreducible Brauer characters and weights.
\end{abstract}

%%%%%%%%%%%%%%%%%%%%%%%%%%%%%%%%%%%%%%%%%%%%%%%%%%%%%%%%%%%%%%%%%

\section{Introduction}\label{sect:intro}

Let $\ell$ be a prime.
For an $\ell$-subgroup $R$ of a finite group $G$, we denote by $\dz(N_G(R)/R)$ the set of all characters of $N_G(R)$ containing $R$ in the kernels and of defect zero as characters of $N_G(R)/R$.
For any $\ell$-block $B$ of $G$, the subset of $\dz(N_G(R)/R)$ consisting of those characters satisfying $\bl(\varphi)^G=B$ is denoted by $\dz(N_G(R)/R,B)$.
Here, $\bl(\varphi)$ is the block of $N_G(R)$ containing $\varphi$.
A weight of $G$ means a pair $(R,\varphi)$ with $\varphi\in\dz(N_G(R)/R)$.
In this case, $R$ is necessarily a radical subgroup of $G$, \emph{i.e.} $R=\rO_\ell(N_G(R))$, and $\varphi$ is called a weight character.
If furthermore $\varphi\in\dz(N_G(R)/R,B)$, then $(R,\varphi)$ is called a $B$-weight.
Denote the set of all $G$-conjugacy classes of $B$-weights by $\cW(B)$.
In the sequel, the term ``weight'' will mean a single weight or a conjugacy class of weights depending on the context.
J. L. Alperin gives the following conjecture in \cite{Al87} called blockwise Alperin weight (BAW) conjecture.

\begin{conjecture}[Alperin]
Let $G$ be a finite group, $\ell$ a prime and $B$ an $\ell$-block of $G$, then $|\cW(B)|=|\IBr(B)|$.
\end{conjecture}

This conjecture has been reduced to the simple groups, which means that the BAW conjecture holds for a finite group $G$ if all non-abelian simple groups involved in $G$ satisfy the inductive blockwise Alperin weight (iBAW) condition.
See \cite{Sp13} for a definition of the iBAW condition and see \cite{KS16} for another version.
Basically speaking, the iBAW condition consists of two parts: the first one requires to establish an equivariant bijection between irreducible Brauer characters and weights; the second one (normally embedded conditions) requires to consider the extension of irreducible Brauer characters and weight characters.
This inductive condition has been verified for several cases, such as: simple alternate groups, many sporadic simple groups, simple groups of Lie type and the defining characteristic, some few cases of simple groups of Lie type and primes different from the defining characteristic; see \cite{Breuer}, \cite{CS13}, \cite{Feng18}, \cite{KS16}, \cite{LZ18}, \cite{Ma14}, \cite{Sch15}, \cite{Sch16}, \cite{Sp13}, etc.

The purpose of this paper is to start the consideration of the iBAW condition for $\PSp_{2n}(q)$ with $q$ a power of an odd prime $p$ and an odd prime $\ell\neq p$.
To do this, we first need to establish a blockwise bijection between the irreducible Brauer characters and weights of the universal covering group $\Sp_{2n}(q)$ of $\PSp_{2n}(q)$.

\begin{mainthm}\label{theorem:main1}
Let $p$ be an odd prime, $q=p^f$ and $\ell$ an odd prime different from $p$, then the blockwise Alperin weight conjecture holds for $\Sp_{2n}(q)$ for $\ell$.
\end{mainthm}

In \cite{FS89}, the classification of blocks and the partition of irreducible ordinary characters into blocks for $\CSp_{2n}(q)$ are obtained.
Using the results of Fong and Srinivasan in \cite{FS89} and the results of Lusztig in \cite{Lu88}, we can classify the blocks of $\Sp_{2n}(q)$ and partition irreducible ordinary characters into blocks.
To give a parametrization of irreducible Brauer characters, we use the basic set $\cE(\Sp_{2n}(q),\ell')$ of $\Sp_{2n}(q)$ given in \cite{Ge93}.
Compared with the parametrization of weights for $\Sp_{2n}(q)$ by J. An in \cite{An94}, the above theorem follows.
In fact, we not only prove the conjecture, but also give an explicit bijection.

\paragraph{}
Then we prove that the above bijecition is equivariant under the action of automorphisms of $\Sp_{2n}(q)$.
For the actions of automorphisms on irreducible characters, we first consider the ordinary characters.
Assuming the unitriagularity of the decomposition matrix with respect to the $\Aut(\Sp_{2n}(q))$-stable basic set $\cE(\Sp_{2n}(q),\ell')$,
the results can be transferred to Brauer characters.
Unfortunately, the unitriangularity of the decomposition matrix has only been proved for linear primes; see \cite{GH97}.
Note that the actions of automorphisms on the irreducible ordinary characters of $\Sp_{2n}(q)$ are also considered by Taylor in \cite{Tay18} using a different method.
We state our next main result as follows.

\begin{mainthm}\label{theorem:main2}
Let $p$ be an odd prime, $q=p^f$ and $\ell$ an odd prime different from $p$.
If the sub-matrix of the decomposition matrix of $\Sp_{2n}(q)$ with respect to the basic set $\cE(\Sp_{2n}(q),\ell')$ is unitriangular, then there is an equivariant bijection between irreducible Brauer characters and weights for $\Sp_{2n}(q)$.
In particular, such an equivariant bijection exists for any linear prime $\ell$.
\end{mainthm}

To consider the actions of automorphisms on weights, we introduced the twisted version of radical subgroups to make the arguments more clear.
We also believe that the twisted version of radical subgroups will be useful  when dealing the extension problem of weight characters, which is required by the ``normally embedded conditions'' in the iBAW condition; an application of twisted version of radical subgroups of type A is in \cite{LZ19}.

In \cite{FLZ19}, the case for type C and linear primes are considered using a different method; they also consider some other cases for classical groups.

\paragraph{}
The paper is organized as follows.
We first fix some notations in section \S\ref{sect:notations}.
Then in section \S\ref{sect:block}, we consider the irreducible characters and blocks of $\Sp_{2n}(q)$.
In section \S\ref{sect:BAWconjecture}, by comparing the parametrization of irreducible Brauer characters with the classification of weights given by J. An in \cite{An94}, we prove the blockwise Alperin weight conjecture for $\Sp_{2n}(q)$(Theorem~\ref{theorem:main1}).
Next, in sections \S\ref{sect:AutCharacter} and \S\ref{sect:AutWeight}, we consider the actions of automorphisms on irreducible (ordinary and Brauer) characters and weights of $\Sp_{2n}(q)$.
As a byproduct, we establish in \S\ref{subsect:equiJordan} for $\Sp_{2n}(q)$ an equivariant Jordan decomposition, which is conjectured by M. Cabanes and B. Sp\"ath for any finite groups of Lie type and verified for type A in \cite[\S8]{CS17A}.
Finally, in section \S\ref{sect:equivariance}, we prove Theorem~\ref{theorem:main2}.

%%%%%%%%%%%%%%%%%%%%%%%%%%%%%%%%%%%%%%%%%%%%%%%%%%%%%%%%%%%%%%%%%
\section{Notations}\label{sect:notations}
%%%%%%%%%%%%%%%%%%%%%%%%%%%%%%%%%%%%%%%%%%%%%%%%%%%%%%%%%%%%%%%%%

In this section, we fix some notations which will be used in this paper; more will be given in the sequel when needed.

%%%%%%%%%%%%%%%%%%%%%%%%%%%%%%%%%%%%%%%%%%%%%%%%%%%%%%%%%%%%%%%%%

\subsection{General notations}
We denote by $C_n$ the cyclic group of order $n$.
Let $G$ be a finite group.
For $H\leqslant G$, the restrictions and inductions of modules or characters are denoted by $\Res^G_H$ and $\Ind_H^G$ respectively.
Let $\ell$ be a prime dividing $G$.
For any $a\in\Aut(G)$ and $\varphi$ an ordinary or Brauer character of $G$, $\varphi^a(g)=\varphi({^ag})$ for any $g\in G$ or $g\in G_{\ell'}$.
The set of all irreducible ordinary or $\ell$-Brauer characters of a finite group $G$ is denoted by $\Irr(G)$ or $\IBr(G)$ respectively.
We also denote by $G^\wedge$ the set of all linear characters of $G$, especially when $G$ is abelian.
Denote by $\Bl(G)$ the set of $\ell$-blocks of $G$.
For any $B\in\Bl(G)$, the set of irreducible or Brauer characters is denoted by $\Irr(B)$ or $\IBr(B)$ respectively.
See, for example, \cite{NT89}, for more notations in the representations of finite groups.

%%%%%%%%%%%%%%%%%%%%%%%%%%%%%%%%%%%%%%%%%%

\subsection{Reductive groups with non connected center}\label{subsect:notationsBon06}
We will need results about reductive groups with non connected centers in \cite{Lu88}.
But for notations, we recall those in \cite{Bon06}.

Let $i:\bG\to\tbG$ be a regular embedding of reductive groups in the sense of \cite{Lu88}.
Then $Z(\tbG)$ is connected, $\tbG=Z(\tbG)\bG$, $Z(\bG)=Z(\tbG)\cap\bG$ and $[\bG,\bG]=[\tbG,\tbG]$.
Let $i^*:\tbG^*\to\bG^*$ be the dual map of $i$, where $\bG^*,\tbG^*$ are the duals of $\bG,\tbG$ respectively.
Then $i^*$ is surjective and $\Ker i^*$ is connected.
Denote by $F$ both the Frobenius maps over $\tbG$ or $\tbG^*$ associated with an $\F_q$-structure whenever no confusion is caused.
Set $G=\bG^F$, $\tG=\tbG^F$, $G^*=\bG^{*F}$ and $\tG^*=\tbG^{*F}$.
The tori $\Ker i^*$ and $\tbG/\bG$ are dual and this dual is compatible with $F$.
Thus there is an isomorphism $(\Ker i^*)^F\cong\Irr(\tG/G)$; for any $z\in(\Ker i^*)^F$, we denote by $\hz$ the irreducible character of $\tG/G$ corresponding to $z$.

For any semisimple element $s$ of $G^*$, set $A_{\bG^*}(s)=C_{\bG^*}(s)/C^\circ_{\bG^*}(s)$.
Let $\Ker' i^*=\Ker i^*\cap[\tbG^*,\tbG^*]$, then there is an injective homomorphism $\varphi_s: A_{\bG^*}(s) \to \Ker' i^*$.
This injection is compatible with the action of $F$, thus induces an injection from $A_{\bG^*}(s)^F$ to $(\Ker' i^*)^F$.
For more details, see \cite[\S8]{Bon06}.

Let $\cZ(\bG)=Z(\bG)/Z^\circ(\bG)$ and $H^1(F,\cZ(\bG))$ be the set of all $F$-conjugacy classes of $\cZ(\bG)$ as in \cite[\S1]{Bon06}.
Since $\cZ(\bG)$ is abelian, we have $H^1(F,\cZ(\bG))=\cZ(\bG)/(F-1)\cZ(\bG)=Z(\bG)/(F-1)Z(\bG)$.
Thus $(\cZ(\bG)^{\wedge})^F\cong H^1(F,\cZ(\bG))^{\wedge}$.
By \cite[\S4]{Bon06}, there is an isomorphism $\omega: \Ker' i^* \to \cZ(\bG)^{\wedge}$ 
which commutes with $F$, thus induces an isomorphism $\omega^0: (\Ker' i^*)^F \to H^1(F,\cZ(\bG))^\wedge$.
The injective homomorphisms obtained by composing $\omega$ and $\omega^0$ with $\varphi_s$ are denoted by $\omega_s$ and $\omega^0_s$ respectively.
For other related notations, see \cite[\S4,\S8]{Bon06}.

On the other hand, for any Levi subgroup $\bL$ of $\bG$, $\tbL=Z(\tbG)\bL$ is a Levi subgroup of $\tbG$ and there is an isomorphism $\sigma_{\bL}^{\bG}: H^1(F,\cZ(\bG)) \to \tbL^F/\bL^FZ(\tbG)^F$.
In particular, it is easy to see that $H^1(F,\cZ(\bG))$ is isomorphic to a subgroup of $\Out(G)$ by setting $
\bL=\bG$.
For details, see \cite[\S6]{Bon06}.

%%%%%%%%%%%%%%%%%%%%%%%%%%%%%%%%%%%%%%%%%%

\subsection{Classical groups}
\emph{In this paper, we assume $q=p^f$ is a power of an odd prime and $\ell$ is an odd prime different from $p$.}
Let $\F_q$ be the finite field of $q$ elements and $\barF_q$ the algebraic closure of $\F_q$.

For a symplectic or orthogonal space $V$ over the field $k=\F_q$ or $\barF_q$, the groups $\I(V)$, $\I_0(V)$, $\J(V)$ and $\J_0(V)$ are defined as in \cite[\S1]{FS89}.
Thus, if $V$ is symplectic, $\Sp(V)=\I(V)=\I_0(V)$ ($\CSp(V)=\J(V)=\J_0(V)$) is the (conformal) symplectic group of $V$;  if $V$ is orthogonal, $\GO(V)=\I(V)$ ($\SO(V)=\I_0(V)$) is the general (special) orthogonal group of $V$.
For an orthogonal space $V$, $\D_0(V)$ denotes the special Clifford group of $V$ as in \cite[\S2]{FS89}.

For any monic polynomial $\Gamma$ in $k[X]$, let $\Gamma^*$ be the monic polynomial whose roots are exactly the inverses of roots of $\Gamma$.
Denote by $\Irr(k[X])$ the set of monic irreducible polynomials.
As in \cite[(1.7)]{FS89}, we set
\begin{align*}
\cF_0 &= \{X-1,X+1\},\\
\cF_1 &= \{ \Gamma \mid \Gamma\in\Irr(k[X]),\Gamma\neq X, \Gamma\neq X\pm1,\Gamma=\Gamma^*\},\\
\cF_2 &= \{ \Gamma\Gamma^* \mid \Gamma\in\Irr(k[X]),\Gamma\neq X, \Gamma\neq X\pm1,\Gamma\neq\Gamma^*\}.
\end{align*}
The polynomials in $\cF=\cF_0\cup\cF_1\cup\cF_1$ serve as the ``elementary divisors'' for semisimple elements in symplectic and orthogonal groups.
For any $\Gamma\in\cF$, the reduced degree $\delta_\Gamma$ and the sign $\varepsilon_\Gamma$ are defined as in \cite[(1.8), (1.9)]{FS89}.

For any semisimple element $s$ of $\I(V)$, we have orthogonal decompositions
$$V=\sum\limits_\Gamma V_\Gamma(s),\ s=\prod\limits_\Gamma s_\Gamma,$$
as in \cite[(1.10)]{FS89}.
Let $m_\Gamma(s)$ be the multiplicity of the elementary divisor $\Gamma$ in $s$.
If $V$ is orthogonal, let $\eta_\Gamma(s)$ be the type of the orthogonal subspace $V_\Gamma(s)$.
Then the multiplicity function $m_\Gamma(s)$ and the type function $\eta_\Gamma(s)$ determines the conjugacy class of $s$ in $\I(V)$.

\paragraph{}
From now on in this paper, $\bV$ ($\bV^*$) will be the symplectic space of dimension $2n$ (the orthogonal space of dimension $2n+1$) over $\barF_q$, and we always set $\bG=\Sp(\bV)$, $\tbG=\CSp(\bV)$, $\bG^*=\SO(\bV)$ and $\tbG^*=\D_0(\bV^*)$ unless otherwise stated.
Then $(\bG,\bG^*)$ and $(\tbG,\tbG^*)$ are pairs of dual reductive groups.
Since $\bV^*$ is of odd dimensional, $\Ker i^*=Z(\tbG^*)$.
Denote by $F$ (or $F_q$) both the standard Frobenius map on $\tbG$ and the dual on $\tbG^*$ defining an $\F_q$-structure on them and by $G,\tG,G^*,\tG^*$ the corresponding groups of fixed points.
These groups of fixed points can be viewed as symplectic or orthogonal or special Clifford groups of a symplectic or orthogonal space $V$ over the finite field $\F_q$; see \cite{FS89}.
Let $D=\langle F_p\rangle$ be the set of field automorphisms,
then $\tG\rtimes{D}$ affords all automorphisms of the simple group $S=G/Z(G)$ of Lie type.

Let $s$ be semisimple element of $G^*=\I_0(V^*)$.
Since $\dim V^*$ is odd, $1$ must be an eigenvalue of $s$ and the conjugacy class $s^{\I_0(V^*)}$ is the same as the conjugacy class $s^{\I(V^*)}$.
For $\Gamma\notin\cF_0$, the type $\eta_\Gamma(s)$ is determined by $m_\Gamma(s)$ (see \cite[(1.12)]{FS89}), thus the conjugacy class of $s$ is determined by the following data: $m_\Gamma(s)$, $\forall \Gamma\in\cF$; $\eta_{X+1}(s)$, $\eta_{X-1}(s)$ satisfying \cite[(1.12)]{FS89}.
In particular, when both $1$ and $-1$ are eigenvalues of $s$, the data $m_\Gamma(s),\Gamma\in\cF$ can not determine the conjugacy class of $s$. 

Let $s$ be a semisimple element of $G^*$.
Then $|A_{\bG^*}(s)|=1$ or $2$ and $A_{\bG^*}(s)^F=A_{\bG^*}(s)$.
In this case, $s$ is isolated if and only if $-1$ is eigenvalues of $s$ (recall that $1$ is always an eigenvalue), and also if and only if $C_{\bG^*}(s)$ is not connected, in which case, $|A_{\bG^*}(s)|=2$.

%%%%%%%%%%%%%%%%%%%%%%%%%%%%%%%%%%%%%%%%%%

\subsection{Cores and quotients}\label{subsect:corequotient}
We recall some concepts and notations concerning partitions, Lusztig symbols and their cores and quotients; see, for example, \cite[Chapter I]{Ol93} and \cite[\S13.8]{Car85}.
Let $e$ be a positive integer (not necessarily prime).

Given a partition $\lambda$ of some natural number $n$, the $e$-core $\lambda_{(e)}$ and $e$-quotient $\lambda^{(e)}$ of $\lambda$ are uniquely determined.
Here, $\lambda^{(e)}=(\lambda_0,\lambda_1,\dots,\lambda_{e-1})$ is an $e$-tuple (ordered sequence) of partitions.
Conversely, given an $e$-core $\kappa$ and an $e$-quotient $(\lambda_0,\lambda_1,\dots,\lambda_{e-1})$, there is a unique partition $\lambda$ with $\lambda_{(e)}=\kappa$ and $\lambda^{(e)}=(\lambda_0,\lambda_1,\dots,\lambda_{e-1})$.
For details, see \cite[\S3]{Ol93}.

The definitions of Lusztig symbols in \cite[\S5]{Ol93} and in \cite[\S13.8]{Car85} are equivalent.
But we will use the one in \cite{Ol93}, where a Lusztig symbol is defined as an unordered pair $\lambda=[X,Y]$ of $\beta$-sets and the pairs $[X,Y]$ and $[X^{+t},Y^{+t}]$ are considered as the same Lusztig symbol; see \cite[\S1]{Ol93} for the definition of $\beta$-sets and the operator $^{+t}$.
The Lusztig symbols of the form $[X,X]$ are called degenerate.
For convenience, \emph{we will view an empty Lusztig symbol $\lambda=\emptyset$ as degenerate}.
Thus when we say $\lambda$ is non degenerate, we have \emph{a fortiori} that $\lambda\neq\emptyset$.

Given a Lusztig symbol $\lambda$, the $e$-core $\lambda_{(e)}$ and $e$-quotient $\lambda^{(e)}$ of $\lambda$ are also uniquely determined.
In \cite{Ol93}, the $e$-core and $e$-quotient is called $e$-cocore and $e$-coquotient respectively if they are obtained by removing the $e$-cohooks, but we will follow \cite[\S9]{FS89} to use the terms cores and quotients also for this case.
As in \cite{Ol93}, $\lambda_{(e)}=[X_0,Y_0]$ is again a Lusztig symbol; $\lambda^{(e)}=[\lambda_0,\dots,\lambda_{e-1};\mu_0,\dots,\mu_{e-1}]$ is an \emph{unordered} pair of two $e$-tuples of partitions, i.e. $[\lambda_0,\dots,\lambda_{e-1};\mu_0,\dots,\mu_{e-1}]$ and $[\mu_0,\dots,\mu_{e-1};\lambda_0,\dots,\lambda_{e-1}]$ are identified.
Conversely, given an $e$-core $[X_0,Y_0]$ and an $e$-quotient $[\lambda_0,\dots,\lambda_{e-1};\mu_0,\dots,\mu_{e-1}]$, there are maybe one or two Lusztig symbols with $[X_0,Y_0]$ and $[\lambda_0,\dots,\lambda_{e-1};\mu_0,\dots,\mu_{e-1}]$ as their core and quotient respectively;
see \cite[pp.39-40]{Ol93}.

%%%%%%%%%%%%%%%%%%%%%%%%%%%%%%%%%%%%%%%%%%%%%%%%%%%%%%%%%%%%%%%%%%%%%%%%%%%%%%%%%%%%

\section{Characters and blocks of $\Sp_{2n}(q)$}\label{sect:block}

In this section, we first give a parametrization of irreducible ordinary characters of $\Sp_{2n}(q)$ from that of $\CSp_{2n}(q)$ using a result of Lusztig about the representations of finite groups of Lie type with disconnected center in \cite{Lu88}.
Then blocks and the partition of $\Irr(G)$ into blocks are obtained from those of $\tG$ by Fong and Srinivasan in \cite{FS89} by considering the block covering.
Finally, a parametrization of irreducible Brauer characters in blocks follows from the basic set $\cE(\Sp_{2n}(q),\ell')$ given in \cite{Ge93}.
Although the results in this section should be well-known for specialists, we state explicitly the results in order to fix the notations in a way which is convenient to our purpose.
The proof about blocks given here uses the constructions in \cite{FS89}; this can probably be obtained using $e$-Jordan-cuspidal pairs in \cite{CE99} and \cite{KM15}.
Note that, the constructions in \cite{FS89} are also used in Proposition \ref{prop:labelweights}.

%%%%%%%%%%%%%%%%%%%%%%%%%%%%%%%%%%%%%%%%%%

\subsection{Irreducible ordinary characters of $\Sp_{2n}(q)$}\label{subsect:irrSp}
We first recall the Jordan decomposition of irreducible ordinary characters of $
\tG=\CSp_{2n}(q)$ and fix some notations.
Let $\ts\in\tG^*$, then we have the following bijections between rational Lusztig series:
$$\cJ: \quad \cE(\tG,\ts) \longleftrightarrow \cE(C_{\tG^*}(\ts),1) \longleftrightarrow \cE(C^\circ_{\bG^*}(s)^F,1),$$
where $s=i^*(\ts)$.
Let $s$ and $V^*$ be orthogonally decomposed as follows:
$$V^*=\sum_\Gamma V^*_\Gamma(s), \quad s=\prod_\Gamma s_\Gamma.$$
Then $C^\circ_{\bG^*}(s)^F\cong \prod\limits_\Gamma C_\Gamma^\circ(s)$ with 
$C_\Gamma^\circ(s):=C_{\I_0(V^*_\Gamma(s))}(s_\Gamma)$ and
$$C_\Gamma^\circ(s)\cong 
\left\{\begin{array}{ll}
\I_0(V^*_\Gamma(s)) & \textrm{if}~ \Gamma\in\cF_0,\\
\GL_{m_\Gamma(s)}(\varepsilon_\Gamma q^{\delta_\Gamma}) & \textrm{if}~ \Gamma\in\cF_1\cup\cF_2
\end{array}\right.$$
(compare with \cite[(1.13)]{FS89}).
Then the unipotent characters $\cE(C^\circ_{\bG^*}(s)^F,1)$ can be parametrized by $\lambda=\prod_\Gamma\lambda_\Gamma$,
where $\lambda_\Gamma$ is a partition of $m_\Gamma(s)$ for $\Gamma\in\cF_1\cup\cF_2$
while $\lambda_\Gamma$ is a Lusztig symbol of rank $[\frac{m_\Gamma(s)}{2}]$ for $\Gamma\in\cF_0$; see \cite[\S13.8]{Car85}.
Recall that each degenerate symbol is counted twice; see \cite[p.132]{FS89} for the definition of the opertator $'$.
The character in $\cE(\tG,\ts)$ corresponding to $\lambda$ is denoted by $\chi_{\ts,\lambda}$.
The label $(\ts,\lambda)$ is determined up to conjugacy in $\tG^*$.

\paragraph{}
When restricting the Jordan decomposition of the previous subsection to $G$, we apply the follow result of Lusztig.
\begin{theorem}[{\cite[Proposition~5.1]{Lu88}}]\label{theorem:JDnonconnectedcenter}
There is a bijection between the orbits of actions of $H^1(F,\cZ(\bG))$ on $\cE(G,s)$ and the orbits of $A_{G^*}(s)^F$ on $\cE(C_{\bG^*}^\circ(s)^F,1)$:
$$\bar{\cJ}: \quad \cE(G,s)/H^1(F,\cZ(\bG)) \longleftrightarrow \cE(C_{\bG^*}^\circ(s)^F,1)/A_{G^*}(s)^F.$$
Furthermore, if $[\varphi]=\bar{\cJ}([\chi])$, then $|[\chi]|=|A_{G^*}(s)^F_\varphi|$.\end{theorem}

The above Jordan decomposition in our case can be made even more explicitly using the partitions and Lusztig symbols.
For a semisimple element $s$ of $G^*$, $\lambda=\prod_\Gamma\lambda_\Gamma$ is as above.
Let $i\Irr(G)$ be the set of $G^*$-conjugacy classes of triples $(s,\lambda,i)$ with $i\in\Z/2\Z$ and \emph{the degenerate symbols are counted only once}.

\begin{theorem}\label{theorem:characterSp}
There is a surjective map:
$$i\Irr(G) \longrightarrow \Irr(G),\quad
(s,\lambda,i)^{G^*} \mapsto \chi_{s,\lambda,i}\in\cE(G,s)$$
satisfying
\begin{compactenum}[(1)]
\item if $(s_1,\lambda_1)$ and $(s_2,\lambda_2)$ are not conjugate under $G^*$, $\chi_{s_1,\lambda_1,i}\neq\chi_{s_2,\lambda_2,j}$;
\item $\chi_{s,\lambda,i}$'s are all the components of $\Res^{\tG}_G\chi_{\ts,\lambda}$;
\item $\chi_{s,\lambda,0}=\chi_{s,\lambda,1}$ (denoted also as $\chi_{s,\lambda}$) if and only if $\lambda_{X+1}$ is degenerate,
	in which case, $\chi_{s,\lambda}=\Res^{\tG}_G\chi_{\ts,\lambda}=\Res^{\tG}_G\chi_{\ts,\lambda'}$.
\end{compactenum}
\end{theorem}

\begin{proof}
The characters in $\cE(G,s)$ are just the components of the restrictions of those in $\cE(\tG,\ts)$.
By the proof of \cite{Lu88}, the bijection $\bar{\cJ}$ is obtained by considering the restriction from $\tG$ to $G$.
Note that $C_{\bG^*}(s)^F\cong C_{X\pm1}(s) \prod\limits_{\Gamma\in\cF_1\cup\cF_2} C_\Gamma^\circ(s)$ with $C_{X\pm1}(s)=\{ g_{X-1}g_{X+1}\in\I_0(V^*_{X-1}(s)\bot V^*_{X+1}(s)) \mid g_{X-\epsilon}\in\I(V^*_{X-\epsilon}(s)),\epsilon=\pm1 \}$.
The actions of $A_{G^*}(s)^F$ follows from \cite[(4B)]{FS89}.
Then the existence of the required bijection follows from Theorem \ref{theorem:JDnonconnectedcenter}.
In particular, if $\lambda_{X+1}\neq\emptyset$ is degenerate, $|A_{G^*}(s)^F_\varphi|=1$ for $\varphi=\chi_{1,\lambda}$ or $\varphi=\chi_{1,\lambda'}$; $\chi_{1,\lambda}$ and $\chi_{1,\lambda'}$ constitute an $A_{G^*}(s)^F$-orbit.
Thus $\Res^{\tG}_G\chi_{\ts,\lambda}=\Res^{\tG}_G\chi_{\ts,\lambda'}$ is an irreducible character of $G$ by Theorem \ref{theorem:JDnonconnectedcenter}.
By \cite[11.6]{Bon06}, the labels of irreducible characters of $\tG$ can be chosen such that $\Res^{\tG}_G \chi_{\ts,\lambda} = \Res^{\tG}_G \chi_{z\ts,\lambda}$ for any $z\in (\Ker i^*)^F$, thus the labels for characters as above are well defined.
\end{proof}

See \cite[TABLE A1]{Wh90} for the Jordan decomposition of characters of $\Sp_4(q)$.

\begin{remark}
Of course, we can adjust the definition of $i\Irr(G)$ slightly to make the above map to be a bijection.
But we choose not to do so, because the above notations are convenient in some cases.
The same remarks apply for parametrization of blocks and irreducible Brauer characters below.
\end{remark}

%%%%%%%%%%%%%%%%%%%%%%%%%%%%%%%%%%%%%%%%%%

\subsection{Blocks of $\Sp_{2n}(q)$}\label{subsect:blockSp}
Let $\tB$ be a blocks of $\tG$ and $B$ is a blocks of $G$ covered by $\tB$.
Assume $\tR$ and $R=G\cap\tR$ be defect groups of $\tB$ and $B$ respectively, thus they are groups of defect type as in \cite[\S5]{FS89}.
We use the notations and results in \cite{FS89}.
In particular, $R=G\cap\tR, \tR=Z(\tG)_\ell R$ and
\begin{align*}
R = R_0 \times R_+ &= R_0 \times R_1 \times\cdots\times R_t, \\
V = V_0 ~\bot~ V_+ &= V_0 ~\bot ~V_1~ \bot\cdots\bot ~V_t,
\end{align*}
where $R_0$ is the trivial group on $V_0$ and $R_i=R_{m_i,\alpha_i}\wr X_{\beta_i}$ is as in \cite[\S5]{FS89} for $i>0$.
Let $C=C_G(R)$ and $\tC=C_{\tG}(R)=C_{\tG}(\tR)$, then
$$C=C_0\times C_+,\quad \tC=\tC_0C_+,$$
where $C_0=\I_0(V_0)$, $\tC_0=\langle\tau,C_0Z(C_+)\rangle$, $C_+=C_1 \times\cdots\times C_t$ with $C_i=C_{\I_0(V_i)}(R_i)=C_{m_i,\alpha_i,\beta_i}=C_{m_i,\alpha_i}\otimes I_{\beta_i}$ and $[\tau,R_+C_+]=1$ with $\tau$ being as in  \cite[\S5]{FS89}.

Let $(R,b),(\tR,\tb)$ be two maximal Brauer pairs of $B,\tB$ respectively and $\theta,\ttheta$ be the canonical characters of $b,\tb$ respectively.
As in \cite{FS89}, these maximal Brauer pairs are denoted as $(R,\theta),(\tR,\ttheta)$.
Then $\theta,\ttheta$ are characters of canonical type as in \cite[\S7]{FS89}.
In particular, there exist decompositions
$$\theta=\theta_0\times\theta_+,\quad \ttheta=\ttheta_0\theta_+,$$
where $\theta_0,\ttheta_0,\theta_+$ are as in \cite[\S7]{FS89}.
Let $N=N_{\I_0(V)}(R)$ and $\tN=N_{\J_0(V)}(\tR)=N_{\J_0(V)}(R)$, then
$$N= N_0\times N_+, \quad \tN= \langle \tau,N \rangle,$$
where $N_0=\I_0(V_0)=C_0$.

\begin{lemma}\label{lemma:MaximalBrauerPairs}
Keep the above notations, $\tB$ covers $1$ or $2$ blocks of $G$ if and only if $\Res^{\tC_0}_{C_0}\ttheta_0$ is an irreducible or a sum of two irreducible characters of $C_0$ respectively.
\end{lemma}

\begin{proof}
The blocks of $G$ with $R$ as a defect group are in bijection with the $N$-conjugacy classes of the canonical characters of $C$.
Since $N_0=C_0$, the $N$-conjugacy class of $\theta$ is $\{ \theta_0\times\theta_+^{n_+} \mid n_+\in N_+\}$.
Then the lemma follows.
\end{proof}

\begin{remark}
\begin{compactenum}[(1)]
\item The first components of $\tN$-conjugates of $\ttheta$ may differ from $\ttheta_0$.
In fact, this is why the second component of the label for a block of $\tG$ given in \cite{FS89} may be a set of cores instead of just a core.
\item The $N$-conjugacy class of $\theta$ contains exactly one normalized canonical 
character $\theta$ (see \cite[p.153]{FS89}).
\end{compactenum}
\end{remark}

Fong and Srinivasan label blocks of $\tG$ by considering some Brauer pairs defined in \cite[\S8]{FS89}.
Let $R'$ be the subgroup obtained from $R$ by replacing each $R_i$ for $i>0$ by the base group $(R_{m_i,\alpha_i})^{r^{\beta_i}}$ of $R_{m_i,\alpha_i}\wr X_{\beta_i}$.
Then
\begin{align*}
R' = R'_0 \times R'_+ &= R'_0 \times R'_1 \times\cdots\times R'_t, \\
V' = V'_0 ~\bot~ V'_+ &= V'_0 ~\bot ~V'_1~ \bot\cdots\bot ~V'_{t'},
\end{align*}
where $V'_0=V_0$, $R'_0=R_0$ and $i'\in i$ denotes $R'_{i'}$ is a component of the base subgroup of $R_i$ as in \cite[\S8]{FS89}.
Let $\tR'=R'Z(\tG)_\ell$ and all the related notations such as $C'$, $\tC'$, $\theta'$, $\ttheta'$ are defined similarly as in \cite[\S8]{FS89}.
In particular,
$$C'=C'_0\times C'_+,\quad \tC'=\tC'_0C'_+,$$
where $C'_0=C_0=\I_0(V_0)$, $\tC'_0=\langle\tau,C_0Z(C'_+)\rangle$, $C'_+=C'_1 \times\cdots\times C'_{t'}$ with $C'_{i'}=C_{m_i,\alpha_i}$ if $i'\in i$.
Hence $Z(C_+)\leqslant Z(C'_+)$.
$\theta'$ and $\ttheta'$ can be decomposed as
$$\theta'=\theta'_0\times\theta'_+, \quad \ttheta'= \ttheta'_0\theta'_+,$$
where $\theta'_0=\theta_0$ and $\ttheta'_0$ is an extension of $\ttheta_0$.
Thus by Lemma \ref{lemma:MaximalBrauerPairs}, the following is immediate.

\begin{lemma}\label{lemma:blockcover}
Keep the above notations, $\tB$ covers $1$ or $2$ blocks of $G$ if and only if $\Res^{\tC'_0}_{C'_0}\ttheta'_0$ is an irreducible or a sum of two 
irreducible characters of $C_0$ respectively.
\end{lemma}

Combined with the above lemma about block covering and the labels of blocks 
of $\CSp_{2n}(q)$ giving in \cite[\S11]{FS89}, we can classify blocks of $G=\Sp_{2n}(q)$.
Let $\ts$ be a semisimple $\ell'$-element of $\tG^*$.
The block of $\tG=\CSp_{2n}(q)$ with label $(\ts,\cK)$ (for the definition of the label, see \cite[\S11]{FS89}, noticing that the first component of the label is denoted as $s$ there) will be denoted as $B_{\ts,\cK}$.
Let $e$ be the multiplicative order of $q^2$ in $(\Z/\ell\Z)^\times$.
Set $e_\Gamma=e$ for $\Gamma\in\cF_0$,  while $e_\Gamma$ is the multiplicative order of $\varepsilon_\Gamma q^{\delta_\Gamma}$ in $(\Z/\ell\Z)^\times$ for $\Gamma\in\cF_1\cup\cF_2$.
The $e_\Gamma$ here is the same as in \cite[\S9]{FS89}.

\emph{Count the degenerate Lusztig symbols only once} and Denote by $i\Bl(G)$ the set of $G^*$-conjugacy classes of the triples $(s,\kappa,i)$ with $s$ a semisimple $\ell'$-element $G^*$, $i\in\Z/2\Z$ and $\kappa= \prod_\Gamma \kappa_\Gamma$ satisfying the condition (C) as below:
\begin{compactenum}[(C.1)]\label{kappa}
\item $\kappa_\Gamma$ the $e_\Gamma$-core of some partition of $m_\Gamma(s)$ for $\Gamma\in\cF_1\cup\cF_2$;
\item $\kappa_{X-1}$ is the $e$-core of some Lusztig symbol of rank $[\frac{m_{X-1}(s)}{2}]$ with odd defect;
\item $\kappa_{X+1}$ is the $e$-core of some Lusztig symbol of rank $\frac{m_{X+1}(s)}{2}$ with defect $\equiv0$ or $2\mod4$ according to $\eta_{X+1}(s)=1$ or $-1$.
\end{compactenum}

\begin{theorem}\label{theorem:blockSp}
There is a surjective map:
$$i\Bl(G) \longrightarrow \Bl(G),\quad (s,\kappa,i)^{G^*} \mapsto B_{s,\kappa,i}$$
satisfying
\begin{compactenum}[(1)]
\item if $(s_1,\kappa_1)$ and $(s_2,\kappa_2)$ are not conjugate under $G^*$, $B_{s_1,\kappa_1,i}\neq B_{s_2,\kappa_2,j}$;
\item $B_{s,\kappa,i}$'s are all the blocks covered by a block $B_{\ts,\cK}$ with $\kappa\in\cK$ (see \cite[pp.172-173]{FS89});
\item $B_{s,\kappa,0}=B_{s,\kappa,1}$ (denoted also as $B_{s,\kappa}$) if and only if $\kappa_{X+1}$ is degenerate,
	in which case, $B_{s,\kappa}$ is covered by the blocks $B_{\ts,\cK}$ and $B_{\ts,\cK'}$ of $\tG$ with $\kappa\in\cK$
	(it may happen that $B_{\ts,\cK}=B_{\ts,\cK'}$; see \cite[pp.172-173]{FS89}).
\end{compactenum}
\end{theorem}

\begin{proof}
Let $\tB$ be a block of $\tG$ with label $(\ts,\cK)$ and a defect group $\tR$ as above.
Keep all the relevant notations as above.
By \cite[(11.2), (11.3)]{FS89},
\begin{align*}
\tC'_0/\I_0(V'_0) &\cong \langle\tau_+,Z(C'_+)\rangle,\\
\tC'_0/Z(C'_+) &\cong \langle\tau_0,\I_0(V'_0)\rangle = \J_0(V'_0)
\end{align*}
As in \cite[p.169]{FS89}, let $\zeta'$ be the linear character of $\tC'_0$ with $\I_0(V'_0)$ in the kernel extending the linear character $\zeta'_+$ induced by $\ttheta'_0$ on $Z(C'_+)$.
Then $\ttheta'_0\zeta'^{-1}$ can be viewed as a character of $\tC'_0/Z(C'_+) \cong \J_0(V'_0)$ which has label $(\ts_0,\kappa)$ and
\begin{equation*}
\Res^{\tC'_0}_{C'_0}\ttheta'_0= \Res^{\tC'_0}_{C'_0}\ttheta'_0\zeta'^{-1}\cong 		\Res^{\tC'_0/Z(C'_+)}_{C'_0Z(C'_+)/Z(C'_+)}\ttheta'_0\zeta'^{-1}\cong 			\Res^{\J_0(V'_0)}_{\I_0(V'_0)}\chi_{\ts_0,\kappa}.
\end{equation*}
Then the number of blocks of $G$ covered by $\tB$ follows from Theorem \ref{theorem:characterSp} and Lemma \ref{lemma:blockcover} and can be labeled as required.

Conversely,  given any triple $(s,\kappa,i)$ with $s\in G^*_{\ell',ss}, i\in \Z/2\Z$, there is an $\ell'$-semisimple element $\ts$ of $\tG^*=\D_0(V^*)$ with $s=i^*(\ts)$ since $G^*$ is a central quotient of $\tG^*$ (\cite[(2.3)]{FS89}).
Let $\tB$ be the block of $\tG$ with label $(\ts,\cK)$, where $\kappa\in\cK$.
Then by the previous paragraph, there is a block of $G$ covered by $\tB$ with label $(s,\kappa,i)$.

Finally, as in Theorem \ref{theorem:characterSp}, the labels are well defined again by 
\cite[11.6]{Bon06}.
\end{proof}

%%%%%%%%%%%%%%%%%%%%%%%%%%%%%%%%%%%%%%%%%%

\subsection{Characters in blocks of $\Sp_{2n}(q)$}\label{subsect:irrB}

Assume $B$ is a block of $G$ with label $(s,\kappa,i)$ and a defect group $R$ 
as before.
Let $\tB$ be a block of $\tG$ covering $B$ with label $(\ts,\cK)$, where $s=i^*(\ts)$ and $\kappa\in\cK$.
Assume $R\neq1$.
We first recall some notations in \cite[\S13]{FS89}.
Let $z$ be an element in $Z(R)$ such that: (i) $z^\ell=1$; (ii) $[V,z_+]=V_+$; (iii) the restriction $z_+$ of $z$ to $V_+$ is primary as an element of $\I_0(V_+)$.
Set $Q=C_G(z)$ and $\tQ=C_{\tG}(z)$.
Then $Q=Q_0\times Q_+$ with $Q_0=\I_0(V_0)$.
Let $\tQ_0=\langle\tau_0,Q_0\rangle=\J_0(V_0)$ and $\tQ_+=\langle\tau_+,Q_+\rangle$.
Then $\tQ$ is the subdirect product of $\tQ_0$ and $\tQ_+$.
Since $z$ is a non-isolated semisimple element of $G$, $Q$ and $\tQ$ are Levi subgroups of $G$ and $\tG$ respectively.
Thus we can choose a dual $\tQ^*$ of $\tQ$ satisfying that $\tC'^*\leqslant \tQ^*\leqslant \tG^*$.
So, $\tQ^*=\tQ^*_0\tQ^*_+$ is the central product of $\tQ^*_0$ and $
\tQ^*_+$ over the kernel $\F_q^\times e$ of the surjective homomorphism $\tG^*=\D_0(V^*)\to G^*=\I_0(V^*)$, where $e$ is the identity element of $\D_0(V^*)$.
By \cite[p.179]{FS89}, $\ts$ can be decomposed as $\ts=\ts_0\ts_+$ with $\ts_0\in\tC'^*_0=\tQ^*_0$ 
and $\ts_+\in\tC'^*_+\leqslant\tQ^*_+$.
A dual defect group of $\tB$ is defined on \cite[p.179]{FS89} as any subgroup of $\tG^*$ of the form $\tR^*=I_0\tR^*_+$ with $I_0$ the identity subgroup of $\D_0(V^*_0)$ and $\tR^*_+$ a Sylow $\ell$-subgroup of $C_{\tQ^*_+}(\ts_+)$.

Now, we introduce some similar notations and definitions for the group $G$ considered in this paper.
Let $Q^*=i^*(\tQ^*)$.Then $Q^*$ is a dual of $Q$ satisfying $C'^*\leqslant Q^*\leqslant G^*$.
Thus $Q^*=Q^*_0\times Q^*_+$ with $Q^*_0=i^*(\tQ^*_0)$ and $Q^*_+=i^*(\tQ^*_+)$; $s$ can be decomposed as $s=s_0\times s_+$, with $s_0=i^*(\ts_0)\in Q^*_0, 
s_+=i^*(\ts_+)\in Q^*_+$.
Then we define a dual defect group of $B$ to be any subgroup of $G^*$ of the form $R^*=i^*(\tR^*)$, thus $R^*=I_0R^*_+$ with $I_0$ the identity subgroup of $\I_0(V^*_0)$ and $R^*_+$ a Sylow $\ell$-subgroup of $C_{Q^*_+}(s_+)$.

The dividing of irreducible characters of $G$ into blocks then follows from \cite[(13C)]{FS89}, Theorem \ref{theorem:characterSp} 
and Theorem \ref{theorem:blockSp}.

\begin{theorem}\label{theorem:irrB}
We can choose the labels of the characters of $G$ such that a character $\chi\in\Irr(G)$ is in $B_{s,\kappa,i}$ if and only if $\chi=\chi_{t,\lambda,j}$, where (i) $t_{\ell'}$ is conjugate to $s$; (ii) $t_\ell$ is in a dual defect group of $B$; (iii) $\kappa$ is the core of $\lambda$; (iv) $j=i$ if $\kappa_{X+1}$ is non degenerate.
\end{theorem}

\begin{remark}
Recall that
\begin{compactenum}[(1)]
\item $\kappa=\prod\limits_\Gamma\kappa_\Gamma$ is the core of $\lambda=\prod\limits_\Gamma\lambda_\Gamma$ means $\kappa_\Gamma$ is the $e_\Gamma$-core of $\lambda_\Gamma$ for each $\Gamma$;
\item when $\kappa_{X+1}$ is non degenerate, $B_{s,\kappa,0}\neq B_{s,\kappa,1}$; if the $e_\Gamma$-core of $\lambda_\Gamma$ is $\kappa_{X+1}$, then $\lambda_{X+1}$ is also non degenerate and $\chi_{t,\lambda,0}\neq\chi_{t,\lambda,1}$, which are in the two blocks $B_{s,\kappa,0}, B_{s,\kappa,1}$ respectively;
\item when $\kappa_{X+1}$ is degenerate, $B_{s,\kappa}=B_{s,\kappa,0}=B_{s,\kappa,1}$, then the part (iv) of the above theorem is superfluous.
\end{compactenum}
\end{remark}

\begin{theorem}\label{theorem:BrauerSp}
Let $B_{s,\kappa,i}$ be a block of $G$ and set
$$i\IBr(B_{s,\kappa,i})=\{(s,\lambda,j)^{G^*} \mid \textrm{$\kappa$ is the core of $\lambda$; $j\in\Z/2\Z$; $j=i$ if $\kappa_{X+1}$ is non degenerate}\}.$$
then there is a surjective map
$$i\IBr(B_{s,\kappa,i}) \longrightarrow \IBr(B_{s,\kappa,i}),\ (s,\lambda,j)\mapsto\varphi_{s,\lambda,j}$$
satisfying that
\begin{compactenum}[(1)]
\item if $\lambda_1\neq\lambda_2$, $\varphi_{s,\lambda_1,j}\neq\varphi_{s,\lambda_2,k}$;
\item $\varphi_{s,\lambda,0}=\varphi_{s,\lambda,1}$ (denoted also as $\varphi_{s,\lambda}$) if and only if $\lambda_{X+1}$ is degenerate, in which case, $\kappa_{X+1}$ is also degenerate.
\end{compactenum}
\end{theorem}

\begin{proof}
By \cite{Ge93}, $\Irr(B_{s,\kappa,i})\cap\cE(G,s)$ is a basic set for $B_{s,\kappa,i}$, then the result follows from Theorem \ref{theorem:irrB}.
\end{proof}

%%%%%%%%%%%%%%%%%%%%%%%%%%%%%%%%%%%%%%%%%%%%%%%%%%%%%%%%%%%%%%%%%%%%%%%%%%%%%%%%%%%%

\section{The blockwise Alperin weight conjecture for $\Sp_{2n}(q)$}\label{sect:BAWconjecture}

In this section, the blockwise Alperin weight conjecture itself will be established for $
\Sp_{2n}(q)$, using the labelling of irreducible Brauer characters in the previous section and the classification of weights in \cite{An94} for $\Sp_{2n}(q)$.

\paragraph{}We first restate the classification of weights for $\Sp_{2n}(q)$ by J. An in \cite{An94}.
Let $\beta_\Gamma=1$ or $2$ according to $\Gamma\in\cF_1\cup\cF_2$ or $\Gamma\in\cF_0$.
\begin{theorem}[{\cite[4F]{An94}}]\label{theorem:paraweights}
Let $B=B_{s,\kappa,i}$ be a block of $G=\Sp_{2n}(q)$.
Assume $s= s_0 \times s_+$ as in \S\ref{subsect:irrB}.
Then $m_\Gamma(s)-m_\Gamma(s_0)=w_\Gamma\beta_\Gamma e_\Gamma$ for some natural number $w_\Gamma$.
Set
\begin{equation}\label{equation:iW(B)withQ}
i\cW(B)=\left\{
Q=\prod_\Gamma Q_\Gamma ~\middle|~
\begin{array}{c}
Q_\Gamma=\left(Q_\Gamma^{(1)},Q_\Gamma^{(2)},\dots,Q_\Gamma^{(\beta_\Gamma e_\Gamma)}\right),\\
\textrm{$Q_\Gamma^{(j)}$'s are partitions},\sum\limits_{j=1}^{\beta_\Gamma e_\Gamma} |Q_\Gamma^{(j)}|=w_\Gamma.
\end{array}
\right\}
\addtocounter{proposition}{1}\tag{\theproposition}
\end{equation}
Here $Q_\Gamma$ is an ordered sequence of $\beta_\Gamma e_\Gamma$ partitions.
Then there is a bijection between $i\cW(B)$ and $\cW(B)$.
The conjugacy class of weights in $B_{s,\kappa,i}$ corresponding to $Q$ will be denoted as $w_{s,\kappa,i,Q}$ ($w_{s,\kappa,0,Q}=w_{s,\kappa,1,Q}$ when $\kappa_{X+1}$ is degenerate, in which case, we also denote this weight as $w_{s,\kappa,Q}$).
\end{theorem}

Note that the decomposition of $s$ in the above theorem is in fact the same as in \cite[(4F)]{An94}; see \cite[p.179]{FS89}, \cite[p.18]{An94} and \S\ref{subsect:irrB}.

\paragraph{}The $Q_{X\pm1}$'s in Theorem~\ref{theorem:paraweights} are ordered sequences of partitions, while the quotients of $\lambda_{X\pm1}$'s in Theorem~\ref{theorem:BrauerSp} are unordered pairs of two $e$-tuples of partitions; see \S\ref{subsect:corequotient}.
To fill the gap between the parametrization of irreducible Brauer characters and that of weights in blocks of $G$, we introduce the notion of \emph{ordered quotients} and \emph{ordered symbols} for Lusztig symbols.

An $e$-quotient $Q=[\lambda_1,\dots,\lambda_e;\mu_1,\dots,\mu_e]$ is said to be degenerate if $(\lambda_1,\dots,\lambda_e) =(\mu_1,\dots,\mu_e)$; otherwise, $Q$ is said to be non degenerate.
Then depending on that $Q$ is degenerate or non degenerate, there are one or two ordered sequences of partitions
$$(\lambda_1,\dots,\lambda_e,\mu_1,\dots,\mu_e),\ (\mu_1,\dots,\mu_e,\lambda_1,\dots,\lambda_e),$$
which are called the \emph{ordered quotients} associated to $Q$ and denoted as $(Q,i),i\in\Z/2\Z$.
When $Q$ is degenerate, $(Q,0)=(Q,1)$.

Similarly, for a Lusztig symbol $\lambda=[X,Y]$, the one or two (depending on that $\lambda$ is degenerate or non degenerate) ordered pairs of $\beta$-sets $(X,Y),(Y,X)$ are called the \emph{ordered symbols} associated with $\lambda$ and denoted as $(\lambda,i),i\in\Z/2\Z$.
When $\lambda$ is degenerate, $(\lambda,0)=(\lambda,1)$.
Here, as in \S\ref{subsect:corequotient}, we identify $(X,Y)$ with $(X^{+t},Y^{+t})$.

Let $\kappa$ be an $e$-core of a Lusztig symbol.
Let $Q$ be an $e$-quotient.
There are one or two Lusztig symbol(s) determined by $\kappa$ and $Q$, which are denoted by $\kappa*(Q,i),i\in\Z/2\Z$.
Thus, when $\kappa$ or $Q$ is degenerate, the above notation reduced to $\kappa*(Q,0)=\kappa*(Q,1)$.

When $\kappa$ is non-degenerate and is an $e$-core of a Lusztig symbol of even defect, we will also introduce the following notations.
Since $\kappa$ is a Lusztig symbol, the ordered symbols $(\kappa,i)$ associated to $\kappa$ are defined.
Let $Q$ be an $e$-quotient.
$\kappa$ and $Q$ determine one or two Lusztig symbols and two or four notations: $(\kappa,i)*(Q,j)$, $i,j\in\Z/2\Z$ depending on that $Q$ is degenerate or non degenerate.
When $Q$ is degenerate, the two notations denote the unique Lusztig symbol determined by $\kappa$ and $Q$.
When $\kappa$ and $Q$ are both non degenerate, we set $(\kappa,i)*(Q,j)=(\kappa,i+1)*(Q,j+1)$, then the two Lusztig symbols determined by $\kappa$ and $Q$ can be denoted as $(\kappa,i)*(Q,j),(\kappa,i')*(Q,j')$ with $i+j\neq i'+j'\mod2$.

\paragraph{}
With the above notions and Theorem~\ref{theorem:BrauerSp}, Theorem~\ref{theorem:paraweights} , Theorem~\ref{theorem:main1} follows with an explicit bijection between the irreducible Brauer characters and weights.

\begin{proposition}\label{prop:bijection}
The blockwise Alperin weight conjecture holds for $\Sp_{2n}(q)$ ($q$ odd) for odd prime $\ell\neq p$.
Specifically, if $B=B_{s,\kappa,i}$ is a block of $G$, we have

\begin{compactenum}[(1)]
\item if $\kappa_{X+1}$ is non degenerate, there is a bijection
$$\IBr(B_{s,\kappa,i}) \longleftrightarrow \cW(B_{s,\kappa,i}),\
\varphi_{s,\lambda,i} \mapsto w_{s,\kappa,i,Q(\lambda,i)},$$
where $Q(\lambda,i)=\prod\limits_\Gamma Q(\lambda,i)_\Gamma$ satisfying
\begin{compactenum}[(i)]
\item if $\Gamma\in\cF_1\cup\cF_2$, $Q(\lambda,i)_\Gamma =\lambda_\Gamma^{(e_\Gamma)}$ is the $e_\Gamma$-quotient of $\lambda_\Gamma$;
\item $Q(\lambda,i)_{X-1}=(\lambda_{X-1}^{(e)},j)$ with $j\in\Z/2\Z$ such that $\lambda_{X-1}=\kappa_{X-1}*(\lambda_{X-1}^{(e)},j)$;
\item $Q(\lambda,i)_{X+1}=(\lambda_{X+1}^{(e)},k)$ with $k\in\Z/2\Z$ such that $\lambda_{X+1}=(\kappa_{X+1},i)*(\lambda_{X+1}^{(e)},k)$;
\end{compactenum}

\item if $\kappa_{X+1}$ is degenerate, $B=B_{s,\kappa}$ and there is a bijection
$$\IBr(B_{s,\kappa}) \longleftrightarrow \cW(B_{s,\kappa}),\
\varphi_{s,\lambda,j} \mapsto w_{s,\kappa,Q(\lambda,j)},$$
where $Q(\lambda,j)=\prod\limits_\Gamma Q(\lambda,j)_\Gamma$ satisfies
\begin{compactenum}[(i)]
\item if $\Gamma\in\cF_1\cup\cF_2$, $Q(\lambda,j)_\Gamma =\lambda_\Gamma^{(e_\Gamma)}$ is the $e_\Gamma$-quotient of $\lambda_\Gamma$;
\item $Q(\lambda,j)_{X-1}=(\lambda_{X-1}^{(e)},k)$ with $k\in\Z/2\Z$ such that $\lambda_{X-1}=\kappa_{X-1}*(\lambda_{X-1}^{(e)},k)$;
\item $Q(\lambda,j)_{X+1}=(\lambda_{X+1}^{(e)},j)$.
\end{compactenum}
\end{compactenum}
\end{proposition}

\begin{proof}
It is clear that the $Q(\lambda,i)$ in (1) and $Q(\lambda,j)$ in (2) satisfies (\ref{equation:iW(B)withQ}), so the maps are well defined.

(1) It suffices to consider the $\Gamma$-parts $\lambda_\Gamma,(Q,i)_\Gamma$ for each $\Gamma$.
Since each partition determines its core and quotient and vice versa, it is clear for $\Gamma\in\cF_1\cup\cF_2$.
For $\Gamma=X-1$, recall that $\kappa_{X-1}$ is always non degenerate.
If $Q$ is a degenerate quotient of which the ordered quotient appears as part of the label of some weight of $B$, then $Q$ together with $\kappa_{X-1}$ determines a unique Lusztig symbol $\kappa_{X-1}*(Q,0)=\kappa_{X-1}*(Q,1)$ which can occur as part of the label of some irreducible Brauer character of $B$.
If $Q$ is a non degenerate quotient such that one of the two ordered quotients appears as part of the label of some weight of $B$, then so does the other one.
Then $Q$ together with $\kappa_{X-1}$ determines two Lusztig symbol $\kappa_{X-1}*(Q,0),\kappa_{X-1}*(Q,1)$, both occuring as parts of the labels of some irreducible Brauer characters of $B$.
The situation for $\Gamma=X+1$ is essentially the same as that of $\Gamma=X-1$, just with a different form.

(2) The situations for $\Gamma\neq X+1$ are the same as in (1).
Assume $\Gamma=X+1$.
If $Q$ is degenerate, the situation is again the same as in (1).
If $Q$ is a non degenerate quotient such that one of the two ordered quotients appears as part of the label of some weight of $B$, then so does the other one.
Since $\kappa_{X+1}$ is degenerate, the two ordered quotients $(Q,0),(Q,1)$ determines unique Lusztig symbol $\kappa_{X-1}*(Q,0)=\kappa_{X-1}*(Q,1)$, but in this case, there are two irreducible Brauer characters $\varphi_{s,\lambda,0}$, $\varphi_{s,\lambda,0}$ corresponding to the two ordered quotients $(Q,0),(Q,1)$.
\end{proof}

\begin{remark}
In part (iii) of (1) of the above thoerem, we choose a different form, the reason for this is that we want to make this bijection equivariant under the action of automorphisms; see Theorem \ref{theorem:equivariance}.
\end{remark}

%%%%%%%%%%%%%%%%%%%%%%%%%%%%%%%%%%%%%%%%%%%%%%%%%%%%%%%%%%%%%%%%%%%%%%%%%%%%%%%%%%%%

\section{Actions of automorphisms on characters}\label{sect:AutCharacter}
In this section, we first calculate the actions of automorphisms on the irreducible ordinary characters of $\Sp_{2n}(q)$ using the parametrization in Theorem \ref{theorem:characterSp}.
Under the unitriangularity hypothesis of the decomposition matrix, the actions of automorphisms on irreducible Brauer characters follows.
Using a different method, Taylor considered part of our result in \cite{Tay18}, but some preparation results in \cite{Tay18} apply to the widest generalization.
We state our result in the form which is convenient to be compared with the actions of automorphisms on weights in the next section;
the proof given here consists of an application of a result of Cabanes and Sp\"ath in \cite{CS13} and some completely elementary arguments.

%%%%%%%%%%%%%%%%%%%%%%%%%%%%%%%%%%%%%%%%%%

\subsection{Actions of automorphisms on characters}

Let $\tau\in\tG$ be an element of $\tG$ generating $\tG$ modulo $G$, then $\tau$ provides the unique diagonal outer automorphism of $G$.
The action of $\tau$ on $\Irr(G)$ immediately follows from Clifford theory and Theorem~\ref{theorem:characterSp}.

\begin{proposition}\label{prop:diagonIrr}
$\chi_{s,\lambda,i}^\tau=\chi_{s,\lambda,i+1}$.
In particular, when $\lambda_{X+1}$ is degenerate, $\tau$ fixes $\chi_{s,\lambda}=\chi_{s,\lambda,0}=\chi_{s,\lambda,1}$.
\end{proposition}

Next, we consider the actions of field automorphisms on irreducible ordinary characters.
We start by stating a result of Cabanes and Sp\"ath to the group $\tG$.

\begin{theorem}[{\cite[Theorem~3.1]{CS13}}]\label{theorem:fieldonIrrCSp}
Let $\sigma$ be a field automorphism on $\tG$ and $\sigma^*$ is the corresponding field automorphism on $\tG^*$ via duality as in \cite[Definition~2.1]{CS13}, then
$$\chi_{\ts,\lambda}^\sigma=\chi_{\sigma^*(\ts),\sigma^*(\varphi_\lambda)},$$
where $\varphi_\lambda$ is the unipotent character of $C_{\tG^*}(\ts)$ corresponding to $\chi_{\ts,\lambda}$ under the Jordan decomposition.
\end{theorem}

We want to derive from the above theorem the actions of field automorphisms on the irreducible ordinary characters of $G$ using the elementary divisors of $s=i^*(\ts)$ and the combinatoric parameters $\lambda$.
To do this, we first consider the action of field automorphisms on the conjufacy classes of semisimple elements of $G^*$, which may be of independent interest.

For any $\Gamma\in\cF$ and a field automorphism $\sigma=F_p^i$,
let $\sigma(\Gamma)$ be the polynomial in $\cF$ whose roots are exactly the $p^i$-th power of the roots of $\Gamma$.
In particular, since $p$ is odd, $\sigma(X+1)=X+1$.

\begin{lemma}\label{lemma:fieldonsemiclass}
Assume $s$ is a semisimple element of $G^*$ and $\sigma=F_p^i$.
Then $m_{\sigma(\Gamma)}(\sigma(s))=m_\Gamma(s)$ and
$\eta_{\sigma(\Gamma)}(\sigma(s))=\eta_\Gamma(s)$.
\end{lemma}

\begin{proof}
Choose a basis $\cB$ of $V^*$ such that the metric matrix of $V^*$ with respect to $\cB$ is $$M=\begin{bmatrix} 0 & \cdots & 1 \\ \vdots & \udots & \vdots \\ 1 & \cdots & 0\end{bmatrix}.$$
Then $G^*$ is isomorphic to $\SO(M):=\{g\in\GL_{2n+1}(q) \mid g^tMg=M,\det(g)=1\}$ and the field automorphism $\sigma$ is realized as:
$g=(g_{jk}) \mapsto (g_{jk}^{p^i})$.

The assertion for the multiplicity function is clear.
Consider next the type function.
Let $V^*$ and $s$ be decomposed orthogonally as follows:
$$V^*= V^*_{\Gamma_1}\perp\cdots\perp V^*_{\Gamma_r},\ s=s_{\Gamma_1}\times\cdots\times s_{\Gamma_r},$$
where $\Gamma_1,\dots,\Gamma_r$ are all the elementary divisors of $s$.
Choose a basis $\cB'=\cB'_1\cup\cdots\cup\cB'_r$ with respect to the above decomposition of $V^*$ and let $A$ be the matrix transferring $\cB$ to $\cB'$.
Let $M_{\Gamma_i}$ be the metric matrix of $V_{\Gamma_i}$ with respect to $\cB'_i$.
Then $A^tMA=\diag\{M_{\Gamma_1},\dots,M_{\Gamma_r}\}$
and $A^{-1}sA=\diag\{s_{\Gamma_1},\dots,s_{\Gamma_r}\}$.
Applying $\sigma$, noticing that $\sigma(M)=M$, we have
\begin{align*}
\sigma(A)^tM\sigma(A) &= \diag\{\sigma(M_{\Gamma_1}),\dots,\sigma(M_{\Gamma_r})\};\\
\sigma(A)^{-1}\sigma(s)\sigma(A) &= \diag\{\sigma(s_{\Gamma_1}),\dots,\sigma(s_{\Gamma_r})\},
\end{align*}
where $\sigma(s_{\Gamma_i})$ is a semisimple element with the unique elementary divisor $\sigma(\Gamma_i)$.
Thus, there are orthogonal decompositions
$$V^*= V^*_{\sigma(\Gamma_1)}\perp\cdots\perp V^*_{\sigma(\Gamma_r)},\ \sigma(s)=\sigma(s)_{\sigma(\Gamma_1)}\times\cdots\times \sigma(s)_{\sigma(\Gamma_r)}.$$
Let $\cB''=\cB''_1\cup\cdots\cup\cB''_r$ be the basis of $V^*$ transferred from $\cB$ by the matrix $\sigma(A)$, where $\cB''_i$ is the basis of $V^*_{\sigma(\Gamma_i)}$ with metric matrix $\sigma(M_{\Gamma_i})$.
Since $\sigma$ does not change the type of the quadratic form of the orthogonal spaces (see \cite[(1.1)$\sim$(1.4)]{FS89}), $\sigma(M_{\Gamma_i})$ is congruent to $M_{\Gamma_i}$.
Thus $\eta_{\sigma(\Gamma_i)}(\sigma(s))=\eta_{\Gamma_i}(s)$ for each $i$.
\end{proof}

\begin{remark}
For $\Gamma\in\cF_1\cup\cF_2$,
$\eta_\Gamma(s)$ is determined by $m_\Gamma(s)$, thus $\eta_{\sigma(\Gamma)}(\sigma(s))=\eta_\Gamma(s)$ can be also derived from the assertion about the multiplicity functions.
The above lemma can be also proved by considering the order of $C_{G^*}(s)$.
But we choose the above more conceptual proof.
\end{remark}

The following result is an analogue of \cite[Lemma~3.2]{LZ18}, but the semisimple element $s$ can be isolated here, so the proof here is different from that in \cite{LZ18}.
\begin{corollary}\label{corollary:fieldonsemiclass}
Assume $s$ is a semisimple element of $G^*$ and $\sigma=F_p^i$.
Then $s^{p^i}$ is conjugate to $\sigma(s)$ in $G^*$.
Consequently, $\langle s\rangle$ and $\langle\sigma(s)\rangle$ are conjugate in $G^*$.
\end{corollary}
\begin{proof}
By Lemma \ref{lemma:fieldonsemiclass}, the multiplicity functions and the type functions of $s^{p^i}$ and $\sigma(s)$ are the same, then the corollary follows.
\end{proof}

For a semisimple element $s$ and $\lambda=\prod\limits_\Gamma\lambda_\Gamma$ with $\lambda_\Gamma$ being a partition of $m_\Gamma(s)$ or a Lusztig symbol of rank $[\frac{m_\Gamma(s)}{2}]$ according to $\Gamma\in\cF_1\cup\cF_2$ or $\Gamma\in\cF_0$,
define $\sigma(\lambda)_{\sigma(\Gamma)}=\lambda_\Gamma$ for a field automorphism $\sigma$ on $G^*$.
By Lemma \ref{lemma:fieldonsemiclass}, this is well-defined.
The first half of the proof of the following theorem is similar to that of \cite[Proposition~3.4]{LZ18}.

\begin{proposition}\label{prop:fieldonirr}
Let $\sigma=F_p^i$ be a field automorphism.
Then we can choose the label $i$'s for irreducible characters $\chi_{s,\lambda,i}$'s of $G$ such that $\chi_{s,\lambda,i}^\sigma=\chi_{\sigma^*(s),\sigma^*(\lambda),i}$.
The choice of $i$'s can be made compatible with that in Theorem \ref{theorem:irrB}.
\end{proposition}

\begin{proof}
Denote $\tchi_{s,\lambda}=\Res^{\tG}_G\chi_{\ts,\varphi_\lambda}$ for $s=i^*(\ts)$; this can be well defined by \cite[11.6]{Bon06} (similarly as in the proof of Theorem \ref{theorem:characterSp}).
Then by Theorem \ref{theorem:fieldonIrrCSp}, $\tchi_{s,\lambda}^\sigma= \Res^{\tG}_G \chi_{\sigma^*(\ts),\sigma^*(\varphi_\lambda)}$.
Since the unipotent characters of $C_{\tG^*}(\ts)$ and $C_{\bG^*}^\circ(s)^F$ are in bijection by restriction, we may assume $\varphi_\lambda\in\cE(C_{\bG^*}^\circ(s)^F,1)$.
By Corollary \ref{corollary:fieldonsemiclass}, we may compose $\sigma^*$ with an inner automorphism of $G^*$, denoted as $\tsigma^*$, such that $\tsigma^*$ stabilizes each component $C_\Gamma^\circ(s)$ of $C_{\tG^*}^\circ(s)^F\cong\prod\limits_\Gamma C_\Gamma^\circ(s)$.
For $\Gamma\in\cF_1\cup\cF_2$, an argument as in the proof of \cite[Proposition~3.4]{LZ18} shows that $\tsigma^*$ stabilizes $\varphi_{\lambda_\Gamma}$.
For $\Gamma\in\cF_0$, $\tsigma^*$ is a field automorphism modular some inner automorphism, thus stabilizes the unipotent characters of $\I_0(V^*_\Gamma(s))$.
Consequently, $\tsigma^*(\varphi_\lambda)=\varphi_\lambda$,
so viewed as unipotent characters of $C_{\tG^*}^\circ(\sigma^*(s))^F$,
$\tsigma^*(\varphi_\lambda)=\varphi_{\tsigma^*(\lambda)}$.
Thus, $\tchi_{s,\lambda}^\sigma= \Res^{\tG}_G \chi_{\sigma^*(\ts),\sigma^*(\varphi_\lambda)}=\tchi_{\tsigma^*(s),\tsigma^*(\lambda)}=\tchi_{\sigma^*(s),\sigma^*(\lambda)}$.

When $\lambda_{X+1}$ is degenerate, the above result is just the required.
So assume $\lambda_{X+1}$ is non degenerate.
In this case, $\tchi_{s,\lambda}=\chi_{s,\lambda,0}+\chi_{s,\lambda,1}$.
It suffices to consider $\sigma=F_p$.
Note that by \cite[Theorem~3.1]{CS17C}, if $F_p^j$ stabilizes $\tchi_{s,\lambda}$, $F_p^j$ stabilizes the both components $\chi_{s,\lambda,0},\chi_{s,\lambda,1}$.
Then taking the $F_p$-orbit of $\tchi_{s,\lambda}$ in $\Irr(G)$, we can choose the subscript $i$ for each one in the orbit of $\tchi_{s,\lambda}$ in $\Irr(G)$ such that $\chi_{s,\lambda,i}^{F_p}=\chi_{F_p^*(s),F_p^*(\lambda),i}$.
Do this for all $F_p$-orbits.

Since $\sigma$ maps all irreducible characters in a block to the set of irreducible characters in another block, the above choice can be made compatible with that in Theorem \ref{theorem:irrB}.
\end{proof}

\begin{proposition}\label{prop:actiononBrauer}
Let $\sigma$ be a field automorphism and $\tau$ be an element of $\tG$ generating $\tG$ modulo $G$.
If the sub-matrix of the decomposition matrix of $G$ with respect to the basic set $\cE(G,\ell')$ is unitriangular,
then we can choose the label for irreducible Brauer characters of $G$ compatible with that in Theorem \ref{theorem:BrauerSp} such that
\begin{compactenum}[(1)]
\item $\varphi_{s,\lambda,i}^\sigma=\varphi_{\sigma^*(s),\sigma^*(\lambda),i}$;
\item $\varphi_{s,\lambda,i}^\tau=\varphi_{s,\lambda,i+1}$.
	In particular, when $\lambda_{X+1}$ is degenerate, $\tau$ fixes $\varphi_{s,\lambda}=\varphi_{s,\lambda,0}=\varphi_{s,\lambda,1}$.
\end{compactenum}
\end{proposition}
\begin{proof}
Note that $\cE(G,\ell')$ is $\Aut(G)$-stable.
By \cite[Lemma~7.5]{CS13}, under the assumption about the decomposition matrix, there is an $\Aut(G)$-equivariant bijection between $\cE(G,\ell')$ and $\IBr(G)$ preserving blocks.
Labelling the irreducible Brauer characters of $G$ by this bijection, the assertions follows. 
\end{proof}

\begin{corollary}\label{corollary:actiononblocks}
Let $\sigma$ and $\tau$ be as above.
Then we can choose the label of blocks such that $B_{s,\kappa,i}^\sigma=B_{\sigma^*(s),\sigma^*(\kappa),i}$ and $B_{s,\kappa,i}^\tau=B_{s,\kappa,i+1}$,
where $\sigma^*(\kappa)_{\sigma^*(\Gamma)}=\kappa_\Gamma$.
\end{corollary}

%%%%%%%%%%%%%%%%%%%%%%%%%%%%%%%%%%%%%%%%%%

\subsection{An equivariant Jordan decomposition for $\Sp_{2n}(q)$}\label{subsect:equiJordan}
Before continuing to consider the actions of automorphisms on weights, we digress to give an equivariant Jordan decomposition for $\Sp_{2n}(q)$, which is conjectured for any finite groups of Lie type in \cite[\S8]{CS17A}.
As in \cite[\S8]{CS17A},
denote $\cE(C_{G^*}(s),1)=\Irr(C_{G^*}(s)\mid\cE(C_{\bG^*}^\circ(s)^F,1))$ and 
$$\Jor(G):=\left(\bigcup_{s\in G^*_{ss}}\cE(C_{G^*}(s),1)\right)/G^*,$$
where the right side is the set of $G^*$-conjugacy classes of pairs of the form $(s,\varphi)$ with $\varphi\in\cE(C_{G^*}(s),1)$.
Then the field automorphism $\sigma$ acts on $\Jor(G)$ naturally via $\sigma^*$.
Recall from \S\ref{subsect:notationsBon06} that we have an injection
$\omega_s^0: A_{\bG^*}(s)^F \to H^1(F,\cZ(\bG))^\wedge$.
Since $H^1(F,\cZ(\bG)) \cong \tbG^F/\bG^FZ(\tbG)^F$, we have a surjection
$$\homega_s^0: \tbG^F/\bG^FZ(\tbG)^F \to (A_{\bG^*}(s)^F)^\wedge.$$
The diagonal automorphism $\tau$ acts on $\Jor(G)$ via $\homega_s^0$:
$\tau.(s,\varphi)=(s,\homega_s^0(\tau)\cdot\varphi)$.

\begin{theorem}\label{theorem:equivariantJD}
Under the above actions of field and diagonal automorphisms, there is an equivariant Jordan decomposition for $G$:
$$\Irr(G) \longleftrightarrow \Jor(G).$$
\end{theorem}

\begin{proof}
By \cite{Lu88} (see Theorem \ref{theorem:JDnonconnectedcenter}), Clifford theory, the actions of automorphisms on $\Irr(G)$ and the actions of automorphisms on $\Jor(G)$ defined above, it suffices to consider $\chi_{s,\lambda,i}$ with $\lambda_{X+1}$ non degenerate on the left side which correspond to some $\varphi^\circ\in\cE(C_{\bG^*}^\circ(s)^F,1)$.
Let $\varphi_0$ and $\varphi_1$ be the two extensions of $\varphi^\circ$ to $C_{G^*}(s)$.
We need to show that the field automorphisms and diagonal automorphisms act on the pairs $(s,\varphi_i)$ on the right side in a similar way as on the left side.
The actions of diagonal automorphisms is clear.

To consider the field automorphisms, we first claim that if a field automorphism $\sigma^*$ of $G^*$ stabilizes $(s,\varphi^\circ)$ with $\varphi^\circ\in\cE(C_{\bG^*}^\circ(s)^F,1)$ and $A_{\bG^*}(s)^F_{\varphi^\circ}=A_{\bG^*}(s)^F$, then $\sigma^*$ stabilizes the two extensions of $\varphi^\circ$ to $C_{G^*}(s)$.
Recall that $C_{\bG^*}^\circ(s)^F=\prod\limits_\Gamma C_\Gamma^\circ(s)$ 
and $C_{G^*}(s)\cong C_{X\pm1}(s) \prod\limits_{\Gamma\in\cF_1\cup\cF_2} C_\Gamma^\circ(s)$ with
$$C_{X\pm1}(s)=\{ g_{X-1}g_{X+1}\in\I_0(V^*_{X-1}(s)\bot V^*_{X+1}(s)) \mid g_{X-\epsilon}\in\I(V^*_{X-\epsilon}(s)),\epsilon=\pm1 \}.$$
Since $m_{X-1}(s)$ is odd, $\I(V^*_{X-1}(s))=\I_0(V^*_{X-1}(s))\times \{\pm 1_{V^*_{X-1}(s)}\}$, thus $\varphi^\circ_{X-1}$ can be extended to a character of $\I(V^*_{X-1}(s))$ stabilized under $\sigma^*$.
Since $\I(V^*_{X+1}(s))$ provides the graph automorphism of $\I_0(V^*_{X+1}(s))$,
$\varphi^\circ_{X+1}$ can be extended to $\I(V^*_{X+1}(s))\rtimes\langle\sigma^*\rangle$ by \cite[Theorem~2.4]{Ma08},
thus \emph{a fortiori}, $\varphi^\circ_{X+1}$ can be extended to a character of $\I(V^*_{X+1}(s))$ stabilized under $\sigma^*$.
Then $\sigma^*$ stabilizes the two extensions of $\varphi^\circ$ to $C_{G^*}(s)$.

By the above claim, the actions of $\tbG^F/\bG^FZ(\tbG)^F$ and field automorphisms commute.
Thus we obtain a well-defined action of $\Out(G)$ on $\Jor(G)$.
Also, by the above claim, for an $F_p$-orbit of pairs $(s,\varphi^\circ)$'s, we can choose the label $i$ such that any field automorphism $\sigma^*$ acts as $\sigma(s,\varphi_i)=(\sigma^*(s),\sigma^*(\varphi)_i)$.
Do this for every $F_p$-orbit.
Then it is easy to construct an $\Aut(G)$-equivariant Jordan decomposition.
\end{proof}

%%%%%%%%%%%%%%%%%%%%%%%%%%%%%%%%%%%%%%%%%%%%%%%%%%%%%%%%%%%%%%%%%%%%%%%%%%%%%%%%%%%%
\section{Actions of automorphisms on weights}\label{sect:AutWeight}

In this section, the actions of automorphisms on weights of $\Sp_{2n}(q)$ will be considered using similar arguments as \cite[\S\S4,5]{LZ18} with some improvements using the twisting process, which is introduced in \cite[\S5]{CS17A} and used in \cite{LZ19}.
This twisting process should be useful when one wants to verify the whole iBAW condition for $\PSp_{2n}(q)$.

%%%%%%%%%%%%%%%%%%%%%%%%%%%%%%%%%%%%%%%%%%

\subsection{Twisted basic subgroups}\label{subsect:twistedbasicsubgp}

We first give a construction of radical subgroups similar as in \cite{LZ18} and \cite{LZ19}, which are slightly different from but conjugate to those in \cite{An94}.
For convenience, we give the complete construction.

Recall that $e$ is the multiplicity order of $q^2$ in $\Z/\ell\Z$.
$\ell$ is said to be linear or unitary if $\ell$ divides $q^e-1$ or $q^e+1$ respectively.
Denote by $\varepsilon$ a sign which is $1$ or $-1$ according to $\ell$ is linear or unitary.
The notation $\F_{\varepsilon{q}}$ will denote the field $\F_q$ or $\F_{q^2}$ according to $\varepsilon=1$ or $-1$.
The notation $\GL(m,-q)$ will denote the general unitary group $\GU(m,q)=\{A\in\GL(m,q^2) \mid F_q(A^t)A=I_m\}$, where $A^t$ denotes the transpose of $A$.

Let $a=v_\ell(q^{2e}-1)$, where $v_\ell$ is the $\ell$-valuation such that $v_\ell(\ell)=1$.
Let $\alpha,\gamma$ be natural numbers.
Denote by $Z_\alpha$ the cyclic group of order $\ell^{a+\alpha}$ and $E_\gamma$ the extraspecial group of order $\ell^{2\gamma+1}$ and exponent $\ell$.
$Z_\alpha E_\gamma$ denotes the central product of $Z_\alpha$ and $E_\gamma$ over $\Omega_1(Z_\alpha)=Z(E_\gamma)$.
Take a set of generators $\{z,x_i,y_i \mid i=1,\dots,\gamma\}$ of $Z_\alpha E_\gamma$ such that $o(z)=\ell^{a+\alpha}$, $o(x_i)=o(y_i)=\ell$ and $[x_i,y_i]=z^{\ell^{a+\alpha-1}}$.

Let $\zeta_\alpha\in\F_{\varepsilon{q}^{e\ell^\alpha}}$ be such that $o(\zeta_\alpha)=\ell^{a+\alpha}$ and $\zeta=\zeta_\alpha^{\ell^{a+\alpha-1}}$.\label{notation:zeta}
Set $Z_\alpha^0=\zeta_\alpha I_\gamma$ with $I_\gamma$ the identity matrix of degree $\ell^\gamma$ and
$$X^0=\diag\{1,\zeta,\cdots,\zeta^{\ell-1}\},\quad
Y^0=\begin{bmatrix} \zero & 1\\ I_{\ell-1} & \zero \end{bmatrix}.$$
Then denote $X_j^0=I_\ell\otimes\cdots\otimes X^0\otimes\cdots\otimes I_\ell$ and $Y_j^0=I_\ell\otimes\cdots\otimes Y^0\otimes\cdots\otimes I_\ell$ with $X_0$ and $Y_0$ appearing as the $j$-th components.
For a positive integer $m$, set $Z_{m,\alpha}^0=I_m\otimes Z_\alpha^0$, $X_{m,j}^0=I_m\otimes X_j^0$ and $Y_{m,j}^0=I_m\otimes Y_j^0$.
Define
\begin{eqnarray*}
\rho_{m,\alpha,\gamma}^0:\quad Z_\alpha E_\gamma 
&\to& \GL(m\ell^\gamma,\varepsilon{q}^{e\ell^\alpha})\\
z &\mapsto& Z_{m,\alpha}^0\\
x_j &\mapsto& X_{m,j}^0\\
y_j &\mapsto& Y_{m,j}^0
\end{eqnarray*}
Set $R_{m,\alpha,\gamma}^0=\rho_{m,\alpha,\gamma}^0(Z_\alpha E_\gamma)$.
Then $R_{m,\alpha,\gamma}^0$ is a subgroup of $\GL(m\ell^\gamma,\varepsilon{q}^{e\ell^\alpha})$ isomorphic to $Z_\alpha E_\gamma$.

Let $M(k)=I_k\otimes\begin{bmatrix} 0 & 1\\ -1 & 0 \end{bmatrix}$, then for any symplectic space $W$ over $\F_q$ of dimension $2k$, there is a basis of $W$ such that the corresponding metric matrix is $M(k)$ and $\Sp(W)\cong\Sp_{M(k)}(2k,q):=\{A\in\GL(2k,q) \mid A^tM(k)A=M(k)\}$.
Similarly, for any symplectic space $\bW$ over $\barF_q$ of dimension $2k$,
we have $\Sp(\bW)\cong\Sp_{M(k)}(2k,\barF_q):=\{A\in\GL(2k,\barF_q) \mid A^tM(k)A=M(k)\}$.
In the sequel, for the symplectic spaces we will encounter, the metric matrices would be of the form $\diag\{M(k_1),\cdots,M(k_r)\}$, which we will just write as $M$ whenever it can be understood from contexts.

Set $\bG_{m,\alpha,\gamma}=\Sp_M(2me\ell^{\alpha+\gamma},\barF_q)$.
An embedding, which can be called a hyperbolic embedding, $\hbar:\GL(m\ell^\gamma,\varepsilon{q}^{e\ell^\alpha}) \to \bG_{m,\alpha,\gamma}$ is defined as:
$$\hbar: A \mapsto \diag\{A,F_q(A),\dots,F_q^{e\ell^\alpha-1}(A),A^{-t},F_q(A^{-t}),\dots,F_q^{e\ell^\alpha-1}(A^{-t})\},$$
where $A^{-t}$ denotes the inverse of the transpose of $A$.
\emph{In the sequel, we will always use $\hbar$ to denote these embeddings of the above type from general linear or unitary groups to symplectic groups of any dimensions.}\label{hyperbolicembedding}
Denote $R_{m,\alpha,\gamma}^{tw}=\hbar(R_{m,\alpha,\gamma}^0)$.
For a natural number, denote by $A_c$ the elementary abelian group of order $\ell^c$.
For any sequence $\bc=(c_1,\dots,c_r)$ of natural numbers, set $A_{\bc}=A_{c_1}\wr\cdots\wr A_{c_r}$ and $|\bc|=c_1+\cdots+c_r$.
Set $R_{m,\alpha,\gamma,\bc}^{tw}=R_{m,\alpha,\gamma}^{tw}\wr A_{\bc}$.

Set $$v_{m,\alpha,\gamma}^0= I_{m\ell^\gamma} \otimes \begin{bmatrix} \zero & 1\\ I_{e\ell^\alpha-1} & \zero \end{bmatrix},\quad
v_{m,\alpha,\gamma}^{(1)}=\diag\{v_{m,\alpha,\gamma}^0,v_{m,\alpha,\gamma}^0\}$$
and 
$$v_{m,\alpha,\gamma}^{(-1)}= I_{m\ell^\gamma} \otimes \begin{bmatrix} \zero & -1\\ I_{2e\ell^\alpha-1} & \zero \end{bmatrix}.$$
Then $v_{m,\alpha,\gamma}^0\in\GL(m\ell^\gamma,\varepsilon{q})$ and $(v_{m,\alpha,\gamma}^0)^{-t}=v_{m,\alpha,\gamma}^0$, thus $v_{m,\alpha,\gamma}^{(1)}\in\bG_{m,\alpha,\gamma}$.
An easy calculation also shows that $v_{m,\alpha,\gamma}^{(-1)}\in\bG_{m,\alpha,\gamma}$.
Set $$v_{m,\alpha,\gamma}=\left\{\begin{array}{ll} v_{m,\alpha,\gamma}^{(1)}, & \textrm{if $\ell$ is linear}\\ v_{m,\alpha,\gamma}^{(-1)}, & \textrm{if $\ell$ is unitary} \end{array}\right.$$
and $v_{m,\alpha,\gamma,\bc}=v_{m,\alpha,\gamma} \otimes I_{\bc}$, where $I_{\bc}$ is the identity matrix of degree $\ell^{|\bc|}$.
Then $v_{m,\alpha,\gamma,\bc}\in \bG_{m,\alpha,\gamma,\bc} :=\Sp_M(2me\ell^{\alpha+\gamma+|\bc|},\barF_q)$.
Thus $R_{m,\alpha,\gamma,\bc}^{tw}$ is a subgroup of $\bG_{m,\alpha,\gamma,\bc}^{v_{m,\alpha,\gamma,\bc}F}$.

By Lang-Steinberg theorem (see for example \cite[Theorem~21.7]{MT11}), there exists an element $g_{m,\alpha,\gamma,\bc}\in\bG_{m,\alpha,\gamma,\bc}$ such that $g_{m,\alpha,\gamma,\bc}^{-1}F(g_{m,\alpha,\gamma,\bc})=v_{m,\alpha,\gamma,\bc}$.
Denote $\iota$ the map $x\mapsto g_{m,\alpha,\gamma,\bc}xg_{m,\alpha,\gamma,\bc}^{-1}$, then $\iota$ induces an isomorphism
$$\iota:\quad \bG_{m,\alpha,\gamma,\bc}^{v_{m,\alpha,\gamma,\bc}F} \to \bG_{m,\alpha,\gamma,\bc}^F,$$
where $\bG_{m,\alpha,\gamma,\bc}^F=\Sp_M(2me\ell^{\alpha+\gamma+|\bc|},q)$ which is often denoted as $G_{m,\alpha,\gamma,\bc}$.
\emph{For the symplectic groups of any dimension, we will always denote the corresponding twisting homomorphisms by the same symbol $\iota$.}\label{twistinghom}
Then $\iota(R_{m,\alpha,\gamma,\bc}^{tw})$ is conjugate to the basic subgroup defined in \cite{An94}.
In this paper, we denote $R_{m,\alpha,\gamma,\bc}:=\iota(R_{m,\alpha,\gamma,\bc}^{tw})$.
When $\bc=(1,\dots,1)$ with $\beta$ one's, $R_{m,\alpha,0,\bc}$ is conjugate to $R_{m,\alpha,\beta}$ in \cite{FS89}.
Then by \cite{An94}, any radical subgroup of $G$ is conjugate to $R_0\times R_1\times\cdots\times R_u$, where $R_0$ is a trivial group and $R_i$ ($i\geqslant1$) is a basic subgroup $R_{m_i,\alpha_i,\gamma_i,\bc_i}$.

Groups such as $\bG_{m,\alpha,\gamma,\bc}^{v_{m,\alpha,\gamma,\bc}F}$ will be called twisted groups.
The constructions in $\bG_{m,\alpha,\gamma,\bc}^{v_{m,\alpha,\gamma,\bc}F}$ corresponding to those in $\bG_{m,\alpha,\gamma,\bc}^F$ via $\iota$ will be called the twisted ones.
For example,  $R_{m,\alpha,\gamma,\bc}^{tw}$ is called a \emph{twisted basic subgroup}.

We want to consider the actions of field and diagonal automorphisms on conjugacy classes of radical subgroups by transferring to twisted groups.
For any positive integer $k$, recall that $\CSp_M(2k,\barF_q):=\{A\in\GL(2k,\barF_q) \mid A^tMA=\lambda M,\lambda\in\barF_q^\times\}$.
Set $\tbG_{m,\alpha,\gamma,\bc}=\CSp_M(2me\ell^{\alpha+\gamma+|\bc|},\barF_q)$.
Let $E=\langle\hF_p\rangle$ be the group of field automorphisms on $\tbG_{m,\alpha,\gamma,\bc}^{v_{m,\alpha,\gamma,\bc}F}$.
As before, denote $\tG_{m,\alpha,\gamma,\bc}:=\tbG_{m,\alpha,\gamma,\bc}^F=\CSp_M(2me\ell^{\alpha+\gamma+|\bc|},q)$.

\begin{lemma}\label{lemma:twist}
$\iota$ can be extended to
$$\iota:\quad \tbG_{m,\alpha,\gamma,\bc}^{v_{m,\alpha,\gamma,\bc}F} \rtimes E \to \tbG_{m,\alpha,\gamma,\bc}^F \rtimes E.$$
Modulo inner automorphisms and via $\iota$, diagonal automorphisms on $\bG_{m,\alpha,\gamma,\bc}^F$ corresponds to diagonal automorphisms on $\bG_{m,\alpha,\gamma,\bc}^{v_{m,\alpha,\gamma,\bc}F}$; field automorphism $\hF_p$ acting on $\bG_{m,\alpha,\gamma,\bc}^F$ corresponds to $\hF_p$ acting on $\bG_{m,\alpha,\gamma,\bc}^{v_{m,\alpha,\gamma,\bc}F}$.
\end{lemma}
\begin{proof}
The extension of $\iota$ can be proved as in \cite[Proposition~5.3]{CS17A}.
The assertion for diagonal automorphisms is clear since the subgroup of diagonal outer automorphisms is of order $2$.
It is routine to see that via $\iota$, $\hF_p$ acting on $\bG_{m,\alpha,\gamma,\bc}^F$ corresponds to $g_{m,\alpha,\gamma,\bc}^{-1}\hF_p(g_{m,\alpha,\gamma,\bc})\hF_p$ ($\hF_p$ composed with the inner automorphism by $g_{m,\alpha,\gamma,\bc}^{-1}\hF_p(g_{m,\alpha,\gamma,\bc})$) acting on $\bG_{m,\alpha,\gamma,\bc}^{v_{m,\alpha,\gamma,\bc}F}$.
Since $\hF_p(v_{m,\alpha,\gamma,\bc})=v_{m,\alpha,\gamma,\bc}$, it follows that $g_{m,\alpha,\gamma,\bc}^{-1}\hF_p(g_{m,\alpha,\gamma,\bc})\in\bG_{m,\alpha,\gamma,\bc}^{v_{m,\alpha,\gamma,\bc}F}$, then the assertion for field automorphisms follows.
\end{proof}

Note that in general $E$ does not act faithfully on $\tbG_{m,\alpha,\gamma,\bc}^F$.
Denote by $D=\langle F_p\rangle$ the group of field automorphisms on $\tbG_{m,\alpha,\gamma,\bc}^F$.
Then $D$ is a quotient group of $E$.

\begin{remark}\label{remark:semidirectRtw}
The twisting process above is similar to that by M. Cabanes and B. Sp\"ath in \cite[\S5]{CS17A}.
Roughly speaking, the twisted constructions are easier to handle than the original ones.
For example, the twisted basic subgroup $R_{m,\alpha,\gamma,\bc}^{tw}$ is of the form $R_t\rtimes R_w$ with $R_t$ a group of diagonal matrices and $R_w$ a group of symmetric matrices.
\end{remark}

Let $\xi$ be an element of $\F_{q^{2e\ell^\alpha}}^\times$ with $o(\xi)=(q-1)(q^{e\ell^\alpha}+1)$.
Then $\xi_0=\xi^{q^{e\ell^\alpha}+1}$ is a generator of $\F_q^\times$.
Set 
\begin{equation}\label{equation:twisteddiag}
\tau_{m,\alpha,\gamma}^{tw}=\left\{ \begin{array}{ll}
I_{me\ell^{\alpha+\gamma}}\otimes\diag\{1,\xi_0\}, & \textrm{if $\varepsilon=1$};\\
I_{m\ell^\gamma}\otimes\diag\{\xi,\xi^q,\cdots,\xi^{q^{e\ell^\alpha-1}};\xi_0\xi^{-1},\xi_0\xi^{-q},\cdots,\xi_0\xi^{-q^{e\ell^\alpha-1}}\}, & \textrm{if $\varepsilon=-1$}.
\end{array}\right.
\addtocounter{proposition}{1}\tag{\theproposition}
\end{equation}
Then $\tau_{m,\alpha,\gamma,\bc}^{tw}=\tau_{m,\alpha,\gamma}^{tw}\otimes I_{\bc}$ generates $\tbG_{m,\alpha,\gamma,\bc}^{v_{m,\alpha,\gamma,\bc}F}$ modulo $\bG_{m,\alpha,\gamma,\bc}^{v_{m,\alpha,\gamma,\bc}F}$, thus provides the unique diagonal outer automorphism.
Let $\tau_{m,\alpha,\gamma,\bc}=\iota(\tau_{m,\alpha,\gamma,\bc}^{tw})$.
Then $\tau_{m,\alpha,\gamma,\bc}$ plays the same roles as $\tau$ in \cite[(1A)]{FS89} and $\tau_{m,\alpha,\beta}$ in \cite[(5D)]{FS89}, etc.

We fix some notations for the centralizers and normalizers of the radical subgroups.
First, denote by $C_{m,\alpha,\gamma}^0$ and $N_{m,\alpha,\gamma}^0$ for the centralizer and normalizer of $R_{m,\alpha,\gamma}^0$ in $\GL(m\ell^\gamma,\varepsilon{q}^{e\ell^\alpha})$.
We denote by $C_{m,\alpha,\gamma,\bc}$, $N_{m,\alpha,\gamma,\bc}$, $\tC_{m,\alpha,\gamma,\bc}$, $\tN_{m,\alpha,\gamma,\bc}$ the centralizers and normalizes of $R_{m,\alpha,\gamma,\bc}$ in $G_{m,\alpha,\gamma,\bc}$ and $\tG_{m,\alpha,\gamma,\bc}$ respectively.
Denote the corresponding twisted constructions by $C_{m,\alpha,\gamma,\bc}^{tw}$, $N_{m,\alpha,\gamma,\bc}^{tw}$, $\tC_{m,\alpha,\gamma,\bc}^{tw}$, $\tN_{m,\alpha,\gamma,\bc}^{tw}$.
When $\bc=\zero$ or $\gamma=0$, we will use obvious abbreviations such as $C_{m,\alpha,\gamma}$, $C_{m,\alpha}$, etc.

\begin{lemma}\label{lemma:twC}
With the above notations, 
\begin{compactenum}[(1)]
\item $C_{m,\alpha}^0=\GL(m,\varepsilon{q}^{e\ell^\alpha})$,
	$C_{m,\alpha,\gamma}^0= C_{m,\alpha}^0 \otimes I_\gamma$,
	$C_{m,\alpha,\gamma,\bc}^{tw} = \hbar(C_{m,\alpha,\gamma}^0) \otimes I_{\bc}$.
\item $\tC_{m,\alpha,\gamma,\bc}^{tw} = \langle C_{m,\alpha,\gamma,\bc}^{tw},\tau_{m,\alpha,\gamma,\bc}^{tw}\rangle$,
	$[\tau_{m,\alpha,\gamma,\bc}^{tw},R_{m,\alpha,\gamma,\bc}^{tw}C_{m,\alpha,\gamma,\bc}^{tw}]=1$,
	$(\tau_{m,\alpha,\gamma,\bc}^{tw})^{q-1} \in Z(C_{m,\alpha,\gamma,\bc}^{tw})$.
\end{compactenum}
\end{lemma}
\begin{proof}
The centralizer of $R_{m,\alpha,\gamma}^0$ in $\GL(m\ell^\gamma,\varepsilon{q}^{e\ell^\alpha})$ follows form \cite[(1B)]{An94}.
The remains then follow from the above constructions; for (2), see also \cite[(5D)]{FS89}.
\end{proof}

\begin{remark}
In \cite{An94}, the author has proved that $C_{m,\alpha,\gamma,\bc} \cong \GL(m,\varepsilon{q}^{e\ell^\alpha}) \otimes I_\gamma\otimes I_{\bc}$.
Here, the meaning of the above lemma is that by transferring to twisted groups, we can have an equality instead of ``$\cong$'', which is useful when we consider the actions of automorphisms and the extensions of weight characters.
\end{remark}

\begin{lemma}\label{lemma:twN}
With the above notations, we have the following.
\begin{enumerate}[(1)]
\item When $\ell$ is linear, $N_{m,\alpha,\gamma}^0=L_{m,\alpha,\gamma}^0R_{m,\alpha,\gamma}^0C_{m,\alpha,\gamma}^0$,
	$[L_{m,\alpha,\gamma}^0,C_{m,\alpha,\gamma}^0R_{m,\alpha,\gamma}^0]=1$,
	$L_{m,\alpha,\gamma}^0 \cap C_{m,\alpha,\gamma}^0R_{m,\alpha,\gamma}^0 = 1$,
	$L_{m,\alpha,\gamma}^0\cong\Sp(2\gamma,\ell)$.
	When $\ell$ is unitary, $N_{m,\alpha,\gamma}^0=L_{m,\alpha,\gamma}^0C_{m,\alpha,\gamma}^0$,
	$[L_{m,\alpha,\gamma}^0,C_{m,\alpha,\gamma}^0]=1$,
	$L_{m,\alpha,\gamma}^0 \cap C_{m,\alpha,\gamma}^0 = Z(L_{m,\alpha,\gamma}^0) = Z(C_{m,\alpha,\gamma}^0)$,
	$L_{m,\alpha,\gamma}^0/R_{m,\alpha,\gamma}^0Z(L_{m,\alpha,\gamma}^0)\cong\Sp(2\gamma,\ell)$.
\item $N_{m,\alpha,\gamma}^{tw}/\hbar(N_{m,\alpha,\gamma}^0)$ is a cyclic group of order $2e\ell^\alpha$.
	Precisely, $N_{m,\alpha,\gamma}^{tw} = \hbar(N_{m,\alpha,\gamma}^0)V_{m,\alpha,\gamma}^{tw}$ with
	$V_{m,\alpha,\gamma}^{tw} = \langle v_{m,\alpha,\gamma}^{(\varepsilon)},M(me\ell^{\alpha+\gamma}) \rangle$.
	When $\ell$ is linear, $V_{m,\alpha,\gamma}^{tw}$ is generated by $v_{m,\alpha,\gamma}^{(1)}M(me\ell^{\alpha+\gamma})$;
	while when $\ell$ is unitary, $V_{m,\alpha,\gamma}^{tw}$ is generated by $v_{m,\alpha,\gamma}^{(-1)}$.
\item $N_{m,\alpha,\gamma,\bc}^{tw} = N_{m,\alpha,\gamma}^{tw}/R_{m,\alpha,\gamma}^{tw} \otimes N_{\fS(\bc)}(A_{\bc})$,
	where $\fS(\bc)$ is the symmetric group of $\ell^{|\bc|}$ symbols and
	$N_{m,\alpha,\gamma}^{tw}/R_{m,\alpha,\gamma}^{tw} \otimes N_{\fS(\bc)}(A_{\bc})$ is defined as \cite[(1.5)]{AF90}.
	Furthermore, $N_{m,\alpha,\gamma,\bc}^{tw}/R_{m,\alpha,\gamma,\bc}^{tw} \cong
	N_{m,\alpha,\gamma,\bc}^{tw}/R_{m,\alpha,\gamma,\bc}^{tw} \times \GL(c_1,\ell)\times\cdots\times\GL(c_r,\ell)$.
\item $\tN_{m,\alpha,\gamma,\bc}^{tw} = \langle N_{m,\alpha,\gamma,\bc}^{tw},\tau_{m,\alpha,\gamma,\bc}^{tw} \rangle$.
\end{enumerate}
\end{lemma}

\begin{proof}
Since $Z(R_{m,\alpha,\gamma}^0)\leqslant Z(\GL(m\ell^\gamma,\varepsilon{q}^{e\ell^\alpha}))$, (1) is just the twisted version of special cases of \cite[(3C)(1)]{AF90} and \cite[(1C)(1)]{An94}.
(2) is just the twisted version of \cite[(1C)(2)]{An94} and \cite[(5B)]{FS89}.
(3) is just the twisted version of \cite[(2.5)]{An94} and (4) is obvious.
\end{proof}

Here, $\hbar(N_{m,\alpha,\gamma}^0)$ is just the twisted version of $N_{m,\alpha,\gamma}^0$ in \cite{An94}.

\begin{lemma}\label{lemma:ActionRadical}
The set $E$ of field automorphisms stabilizes $R_{m,\alpha,\gamma,\bc}^{tw}$ and thus $C_{m,\alpha,\gamma,\bc}^{tw}$, then $N_{\tbG_{m,\alpha,\gamma,\bc}^{vF}\rtimes{E}}(R_{m,\alpha,\gamma,\bc}^{tw}) = \tN_{m,\alpha,\gamma,\bc}^{tw}\rtimes{E} = \langle N_{m,\alpha,\gamma,\bc}^{tw},\tau_{m,\alpha,\gamma,\bc}^{tw} \rangle \rtimes E$.
Consequently, $\Aut(G)$ acts trivially on the conjugacy classes of radical subgroups.
\end{lemma}
\begin{proof}
The first assertion follows from Remark~\ref{remark:semidirectRtw}.
For the second one, first note that $\Aut(G)=\tG/Z(\tG)\rtimes{D}$, where $D$ is the set of field automorphisms.
Let $R=R_0\times R_1\times\cdots\times R_u$ be a radical subgroup of $G$, where $R_0$ is a trivial group and $R_i$ ($i\geqslant1$) is a basic subgroup $R_{m_i,\alpha_i,\gamma_i,\bc_i}$.
Let $v_i=v_{m_i,\alpha_i,\gamma_i,\bc_i}$, $g_i=g_{m_i,\alpha_i,\gamma_i,\bc_i}$ for $i>0$ and $v=1\times v_1\times\cdots\times v_u$, $g=1\times g_1\times\cdots\times g_u$.
Twisted by $v$ and $g$ and consider the twisted groups $\tbG^{vF}\rtimes{E}$.
Set $\tau_i^{tw}=\tau_{m_i,\alpha_i,\gamma_i,\bc_i}^{tw}$ for $i>0$ and $\tau^{tw}=1\times\tau_1^{tw}\times\cdots\times\tau_u^{tw}$.
Assume $R^{tw}$ be the corresponding twisted version of $R$, then by the first assertion, $\tau^{tw}$ and $\hF_p$ fix $R^{tw}$.
Then the lemma follows from Lemma~\ref{lemma:twist}.
\end{proof}

\begin{remark}\label{remark:FieldCGamma}
\begin{compactenum}[(1)]
\item The action of $E$ on $C_{m,\alpha,\gamma,\bc}^{tw}$ induces an action on $C_{m,\alpha}^0$; see Lemma \ref{lemma:twC} and the definition of $\hbar$.
\item The statements in Lemma \ref{lemma:twN} and Lemma \ref{lemma:ActionRadical} demonstrate another advantage of transferring to twisted groups.
\end{compactenum}
\end{remark}

For convenience, we state the following corollary.
\begin{corollary}\label{corollary:twist}
Let $\iota$ be the map induced by conjugacy by $g$ in the proof of Lemma \ref{lemma:ActionRadical}, then it can be extended to
$$\iota:\quad \tbG^{vF} \rtimes E \to \tbG^F \rtimes E,$$
where $E$ is the set of field automorphisms on $\tbG^{vF}$.
Then modulo inner automorphisms, diagonal automorphisms on $\bG^F$ corresponds to diagonal automorphisms on $\bG^{vF}$; field automorphisms acting on $\bG^F$ corresponds to field automorphisms acting on $\bG^{vF}$.
\end{corollary}

%%%%%%%%%%%%%%%%%%%%%%%%%%%%%%%%%%%%%%%%%%

\subsection{A parametrization of weights}\label{subsect:paraweight}

In this subsection, we restate the parametrization of weights of $\Sp_{2n}(q)$ in \cite{An94} in a similar way as \cite{LZ18}.

Let $\cF'$ be the subset of polynomials in $\cF$ whose roots have $\ell'$-order.
Given $\Gamma\in\cF'$, $G_\Gamma$, $R_\Gamma$, $C_\Gamma$, $s_\Gamma$ and $\theta_\Gamma$ are as on \cite[p.22]{An94}.
For convenience, we repeat the constructions for symplectic groups (our notations here are slightly different from those on \cite[p.22]{An94}).
Take integers $m_\Gamma,\alpha_\Gamma$ such that $(m_\Gamma,\ell)=1$ and $m_\Gamma e\ell^{\alpha_\Gamma} = e_\Gamma\delta_\Gamma$.
Here, recall that $\delta_\Gamma$ is the reduced degree as \cite[(1.8)]{FS89}.
Let $V_\Gamma$ be the symplectic space of dimension $2e_\Gamma\delta_\Gamma$ over $\F_q$ and $G_\Gamma=\I_0(V_\Gamma)$.
When a suitable basis for $V_\Gamma$ is chosen, we can identify $G_\Gamma$ with $G_{m_\Gamma,\alpha_\Gamma}$ in the previous subsection.
Then $R_\Gamma=R_{m_\Gamma,\alpha_\Gamma}$ is a basic subgroup of $G_\Gamma$.
Set $C_\Gamma=C_{G_\Gamma}(R_\Gamma)$.
So $C_\Gamma\cong\GL(m_\Gamma,\varepsilon{q}^{e\ell^{\alpha_\Gamma}})$ and has a Coxeter torus $T_\Gamma$ of order $q^{m_\Gamma e\ell^{\alpha_\Gamma}}-\varepsilon^{m_\Gamma}$.
Let $s_\Gamma^*$ be an element of $T_\Gamma$ which is regular as an element of $C_\Gamma$ and as an element of $G_\Gamma$ has a unique elementary divisor $\Gamma$ with multiplicity $\beta_\Gamma e_\Gamma$.
Let $V_\Gamma^*$ be the orthogonal space of dimension $2e_\Gamma\delta_\Gamma+1$ with $\eta(V_\Gamma^*)=1$ and $G_\Gamma^*=\I_0(V_\Gamma^*)$.
Then we can embed the duals $T_\Gamma^*$ and $C_\Gamma^*$ of $T_\Gamma$ and $C_\Gamma$ respectively in $G_\Gamma^*$ such that $T_\Gamma^*\leqslant C_\Gamma^*\leqslant G_\Gamma^*$.
There is a semisimple element $s_\Gamma$ in $T_\Gamma^*$ satisfies: $s_\Gamma$ is a regular element of $C_\Gamma^*$; when viewed as an element of $G_\Gamma^*$, $s_\Gamma$ is dual to $s_\Gamma^*$ in the sense of \cite[(3E)]{An94}. 
We can decompose $V_\Gamma^*$ as follows:
$$V_\Gamma^*=(V_\Gamma^*)_0\bot (V_\Gamma^*)_+,$$
where $\dim(V_\Gamma^*)_0=1$.
The restriction of $s_\Gamma$ to $(V_\Gamma^*)_+$ is primary with the unique elementary divisor $\Gamma$ and multiplicity $\beta_\Gamma e_\Gamma$, while $s_\Gamma$ acts on $(V_\Gamma^*)_0$ trivially.
$\eta((V_\Gamma^*)_+)=\varepsilon_\Gamma^{e_\Gamma}$ or $\varepsilon$ according to $\Gamma\in\cF_1\cup\cF_2$ or $\Gamma\in\cF_0$, while $\eta((V_\Gamma^*)_0)$ is determined by $\eta(V_\Gamma^*)=1$.
Thus $s_\Gamma$ is determined uniquely up to conjugacy in $G_\Gamma^*$.
We may identify $s_\Gamma$ with its restriction to $(V_\Gamma^*)_+$.
Let $\phi_\Gamma$ be the character of $T_\Gamma$ corresponding to $s_\Gamma$ by duality and $\theta_\Gamma = \pm R_{T_\Gamma}^{C_\Gamma}\phi_\Gamma$, where the sign is chosen that $\theta_\Gamma$ is a character of $C_\Gamma$.
Conversely, canonical characters of $\ell$-blocks of $C_{m,\alpha}$ with defect group $R_{m,\alpha}$ are of the form $\theta_\Gamma$ for some $\Gamma$ with $m=m_\Gamma$ and $\alpha_\Gamma$.

Let $R_{\Gamma,\gamma,\bc}= R_{m_\Gamma,\alpha_\Gamma,\gamma,\bc}$ be a basic subgroup and $G_{\Gamma,\gamma,\bc}$,  $C_{\Gamma,\gamma,\bc}$,  $N_{\Gamma,\gamma,\bc}$ be defined similarly.
Then $C_{\Gamma,\gamma,\bc}\cong C_{\Gamma,\gamma}\otimes I_\gamma\otimes I_{\bc}$
and canonical characters of $C_{\Gamma,\gamma,\bc}R_{\Gamma,\gamma,\bc}$ with defect group $R_{\Gamma,\gamma,\bc}$ are all of the form $\theta_{\Gamma,\gamma,\bc}=\theta_\Gamma\otimes I_\gamma\otimes I_{\bc}$.
Then the symplectic space $V_{\Gamma,\gamma,\bc}$ on which $G_{\Gamma,\gamma,\bc}$ acts can be decomposed as
$$V_{\Gamma,\gamma,\bc}=V_\Gamma \bot\cdots\bot V_\Gamma$$
with $\ell^{\gamma+|\bc|}$ terms.
Dually, we have a decomposition
$$V_{\Gamma,\gamma,\bc}^*=(V_{\Gamma,\gamma,\bc}^*)_0\bot (V_{\Gamma,\gamma,\bc}^*)_+,\quad
(V_{\Gamma,\gamma,\bc}^*)_+=(V_\Gamma^*)_+ \bot\cdots\bot (V_\Gamma^*)_+,$$
where $\dim(V_{\Gamma,\gamma,\bc}^*)_0=1$ and $\eta((V_{\Gamma,\gamma,\bc}^*)_0)$ is determined by $\eta(V_{\Gamma,\gamma,\bc}^*)=1$.
\vspace{1em}

Let $\cR_{\Gamma,\delta}$ be the set of all the basic subgroups of the form $R_{\Gamma,\gamma,\bc}$ with $\gamma+|\bc|=\delta$ and we denote $I_\delta=I_\gamma\otimes I_\bc$.
Label the basic subgroups in $\cR_{\Gamma,\delta}$ as $R_{\Gamma,\delta,1}$, $R_{\Gamma,\delta,2}$, $\cdots$ and denote the canonical character associated to $R_{\Gamma,\gamma,i}$ by $\theta_{\Gamma,\gamma,i}$.
It is possible that $m_\Gamma=m_{\Gamma'}=:m$ and $\alpha_\Gamma=\alpha_{\Gamma'}=:\alpha$ for two different $\Gamma,\Gamma'\in\cF'$ and thus $\cR_{\Gamma,\delta}=\cR_{\Gamma',\delta}$.
In this case, as in \cite[\S5]{LZ18}, we make the following convention.

\begin{notation}\label{notation:basicsubgp}
We label the basic subgroups in $\cR_{\Gamma,\delta}$ and $\cR_{\Gamma',\delta}$ such that $R_{\Gamma,\delta,i}=R_{\Gamma',\delta,i}$,
and denote $R_{m,\alpha,\gamma,\bc}$ as $R_{\Gamma,\delta,i}$ or $R_{\Gamma',\delta,i}$ depending on that the related canonical character is $\theta_{\Gamma,\delta,i}$ or $\theta_{\Gamma',\delta,i}$.
\end{notation}

Set $\dz(N_{\Gamma,\delta,i}/R_{\Gamma,\delta,i}\mid\theta_{\Gamma,\delta,i})=
\Irr(N_{\Gamma,\delta,i}\mid\theta_{\Gamma,\delta,i})\cap\dz(N_{\Gamma,\delta,i}/R_{\Gamma,\delta,i})$,
where the meaning of $\dz(N_{\Gamma,\delta,i}/R_{\Gamma,\delta,i})$ is as in \S\ref{sect:intro}.
Set $\sC_{\Gamma,\delta} = \bigcup_i\dz(N_{\Gamma,\delta,i}/R_{\Gamma,\delta,i}\mid\theta_{\Gamma,\delta,i})$.
Assume $\sC_{\Gamma,\delta}=\{\psi_{\Gamma,\delta,i,j}\}$ with $\psi_{\Gamma,\delta,i,j}$ a character of $N_{\Gamma,\delta,i}$.
Then by [An94, (4A)], $|\sC_{\Gamma,\delta}|=\beta_\Gamma e_\Gamma\ell^\delta$.

\paragraph{}
Let $i\cW_\ell(G)$ be the set of $G^*$-conjugacy classes of $(s,\kappa,i,K)$ such that
\begin{compactenum}[(1)]\label{iW(G)}
\item $s$ is an $\ell'$ semisimple element of $G^*$;
\item $\kappa=\prod_\Gamma\kappa_\Gamma$ with $\kappa_\Gamma$ satisfying condition (C) on p.\pageref{kappa};
\item $i\in\Z/2\Z$;
\item $K=\prod_\Gamma K_\Gamma$ with $K_\Gamma:\bigcup_\delta\sC_{\Gamma,\delta} \to \{ \ell\textrm{-cores} \}$ satisfying $\sum_{\delta,i,j}\ell^\delta |K_\Gamma(\psi_{\Gamma,\delta,i,j})|=w_\Gamma$, where $w_\Gamma$ is determined by
\begin{compactenum}[(a)]
\item $m_\Gamma(s)=|\kappa_\Gamma|+e_\Gamma w_\Gamma$ if $\Gamma\in\cF_1\cup\cF_2$;
\item $m_{X+1}(s)=2\rk\kappa_{X+1}+2ew_{X+1}$;
\item $m_{X-1}(s)=(2\rk\kappa_{X-1}+1)+2ew_{X-1}$.
\end{compactenum}
\end{compactenum}
Given a weight $(R,\varphi)$ of $G$, we associate a label $(s,\kappa,i,K)^{G^*}\in i\cW_\ell(G)$ as follows.

 There are corresponding decompositions
 $$ R = R_0 \times R_+,\quad V = V_0 \bot V_+,\quad V^*= (V^*)_0 \bot (V^*)_+,$$
where $R_0$ is the identity group on $V_0$ and $R_+$ is a product of basic subgroups; $(V^*)_0=V_0^*$ is the dual space of $V_0$.
Then $C:=C_G(R)$ and $N:=N_G(R)$ have corresponding decompositions: $C=C_0\times C_+$, $N=N_0\times N_+$ with $C_0=N_0=\I_0(V_0)$.
$\varphi$ lies over a canonical character $\theta$ of $CR$.
Thus $\theta$ can be decomposed as $\theta=\theta_0\times\theta_+$.
Then $\varphi=\varphi_0\times\varphi_+$ with $\varphi_0=\theta_0$ and $\varphi_+\in\Irr(N_+\mid\theta_+)$.

So $\varphi_0$ is a character of defect zero of $\I_0(V_0)$.
Then $\varphi_0=\chi_{s_0,\kappa,i}$, where $s_0$ a semisimple $\ell'$-element of $\I_0(V_0^*)$; $\kappa_\Gamma$ is a partition of $m_\Gamma(s_0)$ or a Lusztig symbol of rank $[\frac{m_\Gamma(s_0)}{2}]$ without hook according to $\Gamma\in\cF_1\cup\cF_2$ or $\Gamma\in\cF_0$; $i\in\Z/2\Z$.
Let $s_0^*\in\I_0(V_0)$ be the dual of $s_0$ in the sense of \cite[p.18]{An94}.

By Notation~\ref{notation:basicsubgp}, $\theta_+$ and $R_+$ can be decomposed as follows:
$$\theta_+=\prod_{\Gamma,\delta,i}\theta_{\Gamma,\delta,i}^{t_{\Gamma,\delta,i}},\quad
R_+=\prod_{\Gamma,\delta,i}R_{\Gamma,\delta,i}^{t_{\Gamma,\delta,i}}.$$
Accordingly, there are decompositions of spaces:
$$V_+=\bot_{\Gamma,\delta,i}V_{\Gamma,\delta,i}^{t_{\Gamma,\delta,i}},\quad
(V^*)_+=\bot_{\Gamma,\delta,i}((V_{\Gamma,\delta,i}^*)_+)^{t_{\Gamma,\delta,i}}.$$
Let $s=s_0\prod_{\Gamma,\delta,i}s_\Gamma\otimes I_\delta\otimes t_{\Gamma,\delta,i}$, where $s_\Gamma$ is a primary semisimple element on $(V_\Gamma^*)_+$ with $m_\Gamma(s_\Gamma)=\beta_\Gamma e_\Gamma$ and $\eta_\Gamma(s_\Gamma)=\varepsilon_\Gamma^{e_\Gamma}$ or $\varepsilon$ according to $\Gamma\in\cF_1\cup\cF_2$ or $\Gamma\in\cF_0$.
Let $s^*=s_0^*\prod_{\Gamma,\delta,i}s_\Gamma^*\otimes I_\delta\otimes t_{\Gamma,\delta,i}$, where $s_\Gamma^*$ is a primary semisimple element on $V_\Gamma$ with $m_\Gamma(s_\Gamma^*)=\beta_\Gamma e_\Gamma$.
Then $s^*\in\I_0(V)$ is the dual of $s$ in the sense of \cite[p.18]{An94}.

Let $\hN_+(\theta_+)=\prod\limits_{\Gamma,\delta,i}N_{\Gamma,\delta,i}\wr\fS(t_{\Gamma,\delta,i})$,
then $N_+(\theta_+)\leqslant\hN_+(\theta_+)$.
Thus $\varphi_+=\Ind_{\hN_+(\theta_+)}^{N_+}\psi_+$ with $\psi_+\in\dz(\hN_+(\theta_+)/R_+\mid\theta_+)$.
Then $\psi_+=\prod\limits_{\Gamma,\delta,i}\psi_{\Gamma,\delta,i}$, where $\psi_{\Gamma,\delta,i}$ is a character of $N_{\Gamma,\delta,i}\wr\fS(t_{\Gamma,\delta,i})$.
By Clifford theory, $\psi_{\Gamma,\delta,i}$ is of the form
\begin{equation}\label{equation:psi}
\Ind_{N_{\Gamma,\delta,i}\wr\prod_j\fS(t_{\Gamma,\delta,i,j})}^{N_{\Gamma,\delta,i}\wr\fS(t_{\Gamma,\delta,i})}
\overline{\prod_j\psi_{\Gamma,\delta,i,j}^{t_{\Gamma,\delta,i,j}}} \cdot \prod_j\phi_{\kappa_{\Gamma,\delta,i,j}},
\addtocounter{proposition}{1}\tag{\theproposition} 
\end{equation}
where $t_{\Gamma,\delta,i}=\sum_jt_{\Gamma,\delta,i,j}$, $\overline{\prod_j\psi_{\Gamma,\delta,i,j}^{t_{\Gamma,\delta,i,j}}}$ is the canonical extension of $\prod_j\psi_{\Gamma,\delta,i,j}^{t_{\Gamma,\delta,i,j}}\in\Irr(N_{\Gamma,\delta,i}^{t_{\Gamma,\delta,i}})$ to $N_{\Gamma,\delta,i}\wr\prod_j\fS(t_{\Gamma,\delta,i,j})$ as in the proof of \cite[Proposition~2.3.1]{Bon99b}, $\kappa_{\Gamma,\delta,i,j}\vdash t_{\Gamma,\delta,i,j}$ without $\ell$-hook and $\phi_{\kappa_{\Gamma,\delta,i,j}}$ the character of $\fS(t_{\Gamma,\delta,i,j})$ corresponding to $\kappa_{\Gamma,\delta,i,j}$.
Note that $\psi_+$, $\psi_{\Gamma,\delta,i}$ and (\ref{equation:psi}) here are slightly different from those on \cite[p.145]{LZ18}.
Now, define $K_\Gamma:\cup_\delta\sC_{\Gamma,\delta}\to\{\ell\textrm{-cores}\}$, $\psi_{\Gamma,\delta,i,j}\mapsto\kappa_{\Gamma,\delta,i,j}$.
Then we can associate the weight $(R,\varphi)$ the label $(s,\kappa,i,K)$.

The below result is implicitly contained in [An94].
\begin{proposition}\label{prop:labelweights}
There is a surjective map:
$$i\cW_\ell(G) \longrightarrow \cW_\ell(G),\quad
(s,\kappa,i,K)^{G^*} \mapsto w_{s,\kappa,i,K}$$
satisfying
\begin{compactenum}[(1)]
\item $w_{s,\kappa,i,K}$ belongs to the block $B_{s,\kappa,i}$;
\item if $(s_1,\kappa_1,K_1)$ and $(s_2,\kappa_2,K_2)$ are not conjugate under $G^*$, $w_{s_1,\kappa_1,i,K_1}\neq w_{s_2,\kappa_2,j,K_2}$;
\item $w_{s,\kappa,0,K}=w_{s,\kappa,1,K}$ (denoted also as $w_{s,\kappa,K}$) if and only if $\kappa_{X+1}$ is degenerate.
\end{compactenum}
\end{proposition}

\begin{proof}
Let $(R,\varphi)$ be a weight with label $(s,\kappa,i,K)$.
We first show that $(R,\varphi)$ belongs to the block with label $(s,\kappa,i)$.
Keep all the above notations.
Let $R'_{\Gamma,\delta,i}= R_\Gamma \times\cdots\times R_\Gamma$ and $\theta'_{\Gamma,\delta,i}= \theta_\Gamma \times\cdots\times \theta_\Gamma$ with $\ell^\delta$ terms.
Set $R'=R_0 \times \prod_{\Gamma,\delta,i}(R'_{\Gamma,\delta,i})^{t_{\Gamma,\delta,i}}$ and $\theta'= \theta_0 \times \prod_{\Gamma,\delta,i}(\theta'_{\Gamma,\delta,i})^{t_{\Gamma,\delta,i}}$.
View $(R,\theta)$ and $(R',\theta')$ as two Brauer pairs.
Then by definition, $(R',\theta')\leqslant(R,\theta)$.
By the results in \S\ref{subsect:blockSp}, $(R',\theta')$ belongs to the block with label $(s,\kappa,i)$, then so do $(R,\theta)$ and the weight $(R,\varphi)$.
By \cite[(1A)]{AF90}, there is a bijection between $\{K\mid(s,\kappa,i,K)^{G^*}\in i\cW_\ell(G)\}$ and $i\cW(B_{s,\kappa,i})$ in equation (\ref{equation:iW(B)withQ}).
Then the proposition follows from \cite[(4F)]{An94} (see Theorem \ref{theorem:paraweights}).
\end{proof}

By the above Proposition, we identify $i\cW(B_{s,\kappa,i})$ in (\ref{equation:iW(B)withQ}) with the set $\{K\mid(s,\kappa,i,K)^{G^*}\in i\cW_\ell(G)\}$.
Of course, we can easily write down the twisted version of the weight characters in this subsection, which we will use in the next two subsections to calculate the actions of automorphisms, and should be useful when one wants to consider the extension problem of weight characters.

%%%%%%%%%%%%%%%%%%%%%%%%%%%%%%%%%%%%%%%%%%

\subsection{Actions of field automorphisms on weights}\label{subsect:fieldonweight}
The aim is to describe the actions of automorphisms on weights using the parametrization in the subsection \S\ref{subsect:paraweight}.
To do this, we transfer to the twisted groups as in \S\ref{subsect:twistedbasicsubgp}.
This is possible by Corollary \ref{corollary:twist}.
We consider the field automorphisms in this subsection and then the diagonal automorphisms in the next.

Given any $\Gamma\in\cF'$, $m_\Gamma,\alpha_\Gamma$ are as before.
Set $R_{\Gamma,\gamma}^0=R_{m_\Gamma,\alpha_\Gamma,\gamma}^0$ be a basic subgroup in $\GL(m_\Gamma\ell^\gamma,\varepsilon{q}^{e\ell^{\alpha_\Gamma}})$, $R_{\Gamma,\gamma,\bc}^{tw}=R_{m_\Gamma,\alpha_\Gamma,\gamma,\bc}^{tw}$ and all the related notations, such that $C_{\Gamma,\gamma}^0,N_{\Gamma,\gamma}^0$, $C_{\Gamma,\gamma,\bc}^{tw},N_{\Gamma,\gamma,\bc}^{tw}$, etc. are defined similarly.
All the results in \S\ref{subsect:twistedbasicsubgp} can be formulated in these notations, such as Lemma \ref{lemma:twC}, Lemma \ref{lemma:twN}, etc.

Recall that $\hbar$ is a hyperbolic embedding as on p.\pageref{hyperbolicembedding}.
Let $T_\Gamma^0$ be a Coxeter torus of $C_\Gamma^0$.
Take $s_\Gamma^0$ to be a regular semisimple element of $T_\Gamma^0$ such that $\hbar(s_\Gamma^0)=\iota^{-1}(s_\Gamma^*)$, where $s_\Gamma^*$ is as in \S\ref{subsect:paraweight} and $\iota$ is the twisting homomorphism as on page \pageref{twistinghom}, i.e. when viewed as an element of $G_\Gamma$, $s_\Gamma^0$ is just $s_\Gamma^*$.
Set $\theta_\Gamma^0=\pm R_{T_\Gamma^0}^{C_\Gamma^0}\hat{s_\Gamma^0}$ and $\theta_{\Gamma,\gamma}^0=\theta_\Gamma^0\otimes I_\gamma$.
Denote by $\theta_{\Gamma,\gamma}^{tw}$ the character of $C_{\Gamma,\gamma}^{tw}$ induced by $\theta_{\Gamma,\gamma}^0$ via $\hbar$ (we sometimes denote as $\theta_{\Gamma,\gamma}^{tw}=\hbar(\theta_{\Gamma,\gamma}^0)$) and $\theta_{\Gamma,\gamma,\bc}^{tw}=\theta_{\Gamma,\gamma}^{tw}\otimes I_\bc$.
Then $\theta_{\Gamma,\gamma,\bc}^{tw}$ is just the twisted version of $\theta_{\Gamma,\gamma,\bc}$ in \S\ref{subsect:paraweight}.

Let $\alpha$ be any natural number.
For any monic polynomial $\Delta$ in $\F_{q^{2e\ell^\alpha}}[X]$, let $\tDelta$ be the monic polynomial in $\F_{q^{2e\ell^\alpha}}[X]$ such that if all roots of $\Delta$ are $\omega_1,\omega_2,\cdots$, then the roots of $\tDelta$ are exactly $\omega_1^{-q^{e\ell^\alpha}},\omega_2^{-q^{e\ell^\alpha}},\cdots$.
As \cite[p.160]{FS89}, we set
\begin{align*}
\cE^\alpha_0 &= \{\Delta \mid \Delta\in\Irr(\F_{q^{e\ell^\alpha}}[X]), \Delta\neq X\},\\
\cE^\alpha_1 &= \{\Delta \mid \Delta\in\Irr(\F_{q^{2e\ell^\alpha}}[X]),\Delta\neq X, \tDelta=\Delta\},\\
\cE^\alpha_2 &= \{\Delta\tDelta \mid \Delta\in\Irr(\F_{q^{2e\ell^\alpha}}[X]),\Delta\neq X, \tDelta\neq\Delta\}
\end{align*}
and
$$\cE^\alpha=\left\{\begin{array}{ll}
\cE^\alpha_0, & \textrm{if $\varepsilon=1$},\\
\cE^\alpha_1\cup\cE^\alpha_2, & \textrm{if $\varepsilon=-1$}.
\end{array}\right.$$

\begin{lemma}\label{lemma:fieldthetaGamma}
Let $\sigma=\hF_p^i$ be a filed automorphism on twisted groups, then $(\theta_{\Gamma,\gamma,\bc}^{tw})^\sigma\equiv\theta_{\sigma(\Gamma),\gamma,\bc}^{tw}$ modulo the conjugacy of $N_{\Gamma,\gamma,\bc}^{tw}=N_{\sigma(\Gamma),\gamma,\bc}^{tw}$.
Furthermore, $(\theta_{\Gamma,\gamma,\bc}^{tw})^\sigma$ is conjugate to $\theta_{\sigma(\Gamma),\gamma,\bc}^{tw}$ by an element of $V_{\Gamma,\gamma,\bc}^{tw}$ (see Lemma \ref{lemma:twN}).
\end{lemma}
\begin{proof}
It suffices to consider $\theta_\Gamma^{tw}=\hbar(\theta_\Gamma^0)$.
By Remark \ref{remark:FieldCGamma}, the action of $\sigma$ on $C_\Gamma^{tw}$ induces an action on $C_\Gamma^0$.
Then by \cite[Proposition 3.4]{LZ18}, $(\theta_\Gamma^0)^\sigma= \pm R_{T_\Gamma^0}^{C_\Gamma^0}\widehat{\sigma(s_\Gamma^0)}$.
As an element of $C_\Gamma^0=\GL(m_\Gamma,\varepsilon{q}^{e\ell^{\alpha_\Gamma}})$, $s_\Gamma^0$ has only one elementary divisor $\Delta\in\cE^{\alpha_\Gamma}$ with multiplicity one.
Thus as an element of $C_\Gamma^0$, $\sigma(s_\Gamma^0)$ has only one elementary divisor $\sigma(\Delta)$ with multiplicity one.
Since all roots of $\Delta$ are roots of $\Gamma$, all roots of $\sigma(\Delta)$ are roots of $\sigma(\Gamma)$.
So as an element of $G_\Gamma$, $\sigma(\hbar(s_\Gamma^0))$ has $\sigma(\Gamma)$ as its unique elementary divosor.
Thus $(\theta_{\Gamma}^{tw})^\sigma\equiv\theta_{\sigma(\Gamma)}^{tw}$.
The last assertion follows from the structure of $N_{\Gamma,\gamma,\bc}^{tw}$.
\end{proof}

\begin{remark}\label{remark:normalizedtheta}
By the above lemma, for each $\Gamma\in\cF$ and a filed automorphism $\sigma=\hF_p^i$ fixing $\Gamma$,
we can compose $\sigma$ with the conjugate by an element of $V_{\Gamma,\gamma,\bc}^{tw}$, which we denoted as $\tsigma$, such that $(\theta_{\Gamma,\gamma,\bc}^{tw})^{\tsigma}=\theta_{\Gamma,\gamma,\bc}^{tw}$.
\end{remark}

Now, we can give the description of $\dz(N_{\Gamma,\gamma}^{tw}/R_{\Gamma,\gamma}^{tw}\mid\theta_{\Gamma,\gamma}^{tw})$ whose untwisted version is given in \cite{An94}.
When $\ell$ is linear, by (1) of Lemma \ref{lemma:twN}, $\vartheta_{\Gamma,\gamma}^0:=\St_\gamma\theta_{\Gamma,\gamma}^0$ is the unique character of $\dz(N_{\Gamma,\gamma}^0/R_{\Gamma,\gamma}^0\mid\theta_{\Gamma,\gamma}^0)$.
Assume $\ell$ is unitary.
Recall from (1) of Lemma \ref{lemma:twN} that $N_{\Gamma,\gamma}^0$ is a central product of $L_{\Gamma,\gamma}^0$ and $C_{\Gamma,\gamma}^0$ over $Z(L_{\Gamma,\gamma}^0)=Z(C_{\Gamma,\gamma}^0)$.
$\theta_{\Gamma,\gamma}^0$ induces a linear character of $Z(L_{\Gamma,\gamma}^0)$, which by \cite[(1B)]{An94} can be extended to a linear character $\xi_{\Gamma,\gamma}^0$ of $L_{\Gamma,\gamma}^0$ trivial on $R_{\Gamma,\gamma}^0$.
Let $\St_\gamma$ be the Steinberg character of $L_{m,\alpha,\gamma}^0/R_{m,\alpha,\gamma}^0Z(L_{m,\alpha,\gamma}^0)\cong\Sp(2\gamma,\ell)$.
Then $\vartheta_{\Gamma,\gamma}^0:=\St_\gamma\xi_{\Gamma,\gamma}^0\theta_{\Gamma,\gamma}^0$ is the unique character of $\dz(N_{\Gamma,\gamma}^0/R_{\Gamma,\gamma}^0\mid\theta_{\Gamma,\gamma}^0)$ when $\ell$ is unitary; see the proof of \cite[(3A)]{An94}.
Denote by $\vartheta_{\Gamma,\gamma}^{tw}:=\hbar(\vartheta_{\Gamma,\gamma}^0)$ the character of $\hbar(N_{\Gamma,\gamma}^0)$ induced by $\vartheta_{\Gamma,\gamma}^0$ via $\hbar$.\label{vartheta}
Then $\vartheta_{\Gamma,\gamma}^{tw}$ is the unique character of $\dz\left(\hbar(N_{\Gamma,\gamma}^0)/R_{\Gamma,\gamma}^{tw}\mid\theta_{\Gamma,\gamma}^{tw}\right)$.
Thus $N_{\Gamma,\gamma}^{tw}(\theta_{\Gamma,\gamma}^{tw}) =N_{\Gamma,\gamma}^{tw}(\vartheta_{\Gamma,\gamma}^{tw}) =\hbar(N_{\Gamma,\gamma}^0)(V_{\Gamma,\gamma}^{tw})_{\theta_{\Gamma,\gamma}^{tw}}$, which is abbreviated as $N(\theta_{\Gamma,\gamma}^{tw})=N(\vartheta_{\Gamma,\gamma}^{tw})$, where $V_{\Gamma,\gamma}^{tw}$ is as in (2) of Lemma \ref{lemma:twN}.
By \cite[(3I)]{An94}, $N(\theta_{\Gamma,\gamma}^{tw})/\hbar(N_{\Gamma,\gamma}^0)$ is cyclic of order $\beta_\Gamma e_\Gamma$ and $\dz\left(N_{\Gamma,\gamma}^{tw}/R_{\Gamma,\gamma}^{tw}\mid\theta_{\Gamma,\gamma}^{tw}\right)$ has exactly $\beta_\Gamma e_\Gamma$ characters all of which lie over $\vartheta_{\Gamma,\gamma}^{tw}$ (note that $(\beta_\Gamma e_\Gamma, \ell)=1$). 

\begin{lemma}\label{lemma:extensiontheta}
There are exactly $\beta_\Gamma e_\Gamma$ extension of $\theta_{\Gamma,\gamma}^{tw}$ to $C_{\Gamma,\gamma}^{tw}(V_{\Gamma,\gamma}^{tw})_{\theta_{\Gamma,\gamma}^{tw}}$.
There is a bijection between $\dz\left(N_{\Gamma,\gamma}^{tw}/R_{\Gamma,\gamma}^{tw}\mid\theta_{\Gamma,\gamma}^{tw}\right)$ and $\dz\left(C_{\Gamma,\gamma}^{tw}(V_{\Gamma,\gamma}^{tw})_{\theta_{\Gamma,\gamma}^{tw}}/Z(R_{\Gamma,\gamma}^{tw})\mid\theta_{\Gamma,\gamma}^{tw}\right)$, which is equivariant under $\tau_{\Gamma,\gamma}^{tw}$ and the $\tsigma$ as in Remark \ref{remark:normalizedtheta}.
\end{lemma}
\begin{proof}
For simplicity of notations, we set in this proof that $R=R_{\Gamma,\gamma}^{tw}$, $L=\hbar(L_{\Gamma,\gamma}^0)$, $C=C_{\Gamma,\gamma}^{tw}=\hbar(C_{\Gamma,\gamma}^0)$, $V=V_{\Gamma,\gamma}^{tw}$, $N=N_{\Gamma,\gamma}^{tw}=LCV$; $\theta=\theta_{\Gamma,\gamma}^{tw}$, $\theta_1=\hbar(\St_\gamma)$ or $\hbar(\St_\gamma\xi_{\Gamma,\gamma}^0)$ according to $\ell$ is linear or unitary, $\vartheta=\vartheta_{\Gamma,\gamma}^{tw}=\theta_1\theta$; $N(\theta)=N(\vartheta)=LCV_\theta$.

The first assertion follows from that $CV_\theta/C$ is a cyclic group of order $\beta_\Gamma e_\Gamma$.
For the second assertion, we first note that by Clifford theory, there is a bijection between $\dz(N/R\mid\theta)$ and $\dz(N(\theta)/R\mid\theta)$.
By Lemma \ref{lemma:twN}, $V$ also normalizes $L$.
Since $LC$ is a central product and $\vartheta=\theta_1\theta$ is the unique character in $\dz(LC/R\mid\theta)$, $V_\theta$ also fixes $\theta_1$ (but may not fix $\hbar(\xi_{\Gamma,\gamma}^0)$).
View $N(\theta)=LCV_\theta$ as a subgroup of $(LV_\theta)(CV_\theta)$.
By the definition of $\theta_1$, we can fix an extension $\ttheta_1$ of $\theta_1$ to $LV_\theta$ which is stable under $\tau_{\Gamma,\gamma}^{tw}$ and $\tsigma$ (when $\ell$ is linear, note the structure of $N_{\Gamma,\gamma}^0$ from (1) of Lemma \ref{lemma:twN}; when $\ell$ is unitary, note that the set of graph automorphism and field automorphisms is cyclic).
Then for any extension $\ttheta$ of $\theta$ to $CV_\theta$,
$\Res^{(LV_\theta)(CV_\theta)}_{LCV_\theta}\ttheta_1\ttheta$ is an extension of $\vartheta$.
The map $\ttheta\mapsto\Res^{(LV_\theta)(CV_\theta)}_{LCV_\theta}\ttheta_1\ttheta$ is a bijection between $\dz(N(\theta)/R\mid\theta)$ and $\dz(CV_\theta/Z(R)\mid\theta)$.
Combining the above two bijiections, we have a bijection between $\dz(N/R\mid\theta)$ and $\dz(CV_\theta/Z(R)\mid\theta)$.
The equivariance follows from the construction of the bijection.
\end{proof}

\begin{lemma}\label{lemma:fieldonextensiontheta}
Let $\tsigma$ be as in Remark \ref{remark:normalizedtheta}.
Then $\theta_{\Gamma,\gamma}^{tw}$ extends to $C_{\Gamma,\gamma}^{tw}\langle(V_{\Gamma,\gamma}^{tw})_{\theta_{\Gamma,\gamma}^{tw}}, \tsigma\rangle$.
In particular, the $\beta_\Gamma e_\Gamma$ extensions of $\theta_{\Gamma,\gamma}^{tw}$ to $C_{\Gamma,\gamma}^{tw}(V_{\Gamma,\gamma}^{tw})_{\theta_{\Gamma,\gamma}^{tw}}$ is stabilized under $\tsigma$.
\end{lemma}
\begin{proof}
It suffices to prove the first assertion.
Denote in this proof that $v=v_{\Gamma,\gamma}^{(\varepsilon)}$ and $M=M(m_\Gamma e\ell^{\alpha_\Gamma})$.
$\tsigma$ commutes with $v$ and $M$.
Recall that $C_{\Gamma,\gamma}^{tw}\cong\GL(m_\Gamma\ell^\gamma,\varepsilon{q}^{e\ell^{\alpha_\Gamma}})$, on which $\tsigma,v,M$ act as field-graph automorphisms.
Then the assertion follows from \cite[4.3.2]{Bon99a} and \cite[Remark~9.3(a)]{Sp09};
the details are similar to the proof of \cite[4.17]{Feng18}.
\end{proof}

\begin{lemma}\label{lemma:fieldondz}
Let $\tsigma$ be as in Remark \ref{remark:normalizedtheta}.
Then $\tsigma$ acts trivially on the set $\dz(N_{\Gamma,\gamma,\bc}^{tw}/R_{\Gamma,\gamma,\bc}^{tw}\mid\theta_{\Gamma,\gamma,\bc}^{tw})$.
\end{lemma}
\begin{proof}
The assertion for $\bc=\zero$ follows from Lemma \ref{lemma:extensiontheta} and Lemma \ref{lemma:fieldonextensiontheta}.
Then assume $\bc\neq\zero$.
Noting that $N_{m,\alpha,\gamma,\bc}^{tw}/R_{m,\alpha,\gamma,\bc}^{tw} \cong N_{m,\alpha,\gamma,\bc}^{tw}/R_{m,\alpha,\gamma,\bc}^{tw} \times \GL(c_1,\ell)\times\cdots\times\GL(c_r,\ell)$ by (3) of Lemma \ref{lemma:twN} and field automorphisms commute with each $\GL(c_i,\ell)$, then the general case follows from the case $\bc=\zero$.
\end{proof}

By Remark \ref{remark:normalizedtheta} and Lemma \ref{lemma:fieldondz}, we can make the following convention on notations.
\begin{notation}\label{notation:field}
Let $\cR_{\Gamma,\delta}^{tw}=\{R_{\Gamma,\delta,i}^{tw}\}$ and $\sC_{\Gamma,\delta}^{tw}=\{\psi_{\Gamma,\delta,i,j}^{tw}\}$ be the twisted version of $\cR_{\Gamma,\delta}$ and $\sC_{\Gamma,\delta}$ respectively.
Let $\sigma=\hF_p^i$ be a field automorphism on the twisted groups.
Then we can choose the notations such that $(R_{\Gamma,\delta,i}^{tw})^\sigma=R_{\sigma(\Gamma),\delta,i}^{tw}$ and $(\psi_{\Gamma,\delta,i,j}^{tw})^\sigma=\psi_{\sigma(\Gamma),\delta,i,j}^{tw}$.
\end{notation}

\begin{remark}
In \cite[Convention 5.2]{LZ18}, we make the similar convention on notations as in Notation \ref{notation:field} without further explanation.
The convention in \cite{LZ18} is well-defined by a similar result as Lemma \ref{lemma:fieldondz}.
The proof of such a result in the case of \cite{LZ18} is similar to the case here.
\end{remark}

\begin{proposition}\label{prop:fieldonweight}
Let $(R,\varphi)$ be a weight of $G$ with the label $(s,\kappa,i,K)^{G^*}$ (or the corresponding label $(s,\kappa,i,Q)^{G^*}$) and $\sigma$ be a field automorphism,
then $(R,\varphi)^\sigma$ has label $(\sigma^*(s),\sigma^*(\kappa),i,\sigma^*(K))^{G^*}$ (or the corresponding label $(\sigma^*(s),\sigma^*(\kappa),i,\sigma^*(Q))^{G^*}$).
Here, $\sigma^*(\kappa)$ is as before and $\sigma^*(K)_{\sigma^*(\Gamma)}=K_\Gamma$, $\sigma^*(Q)_{\sigma^*(\Gamma)}=Q_\Gamma$.
\end{proposition}
\begin{proof}
By Corollary \ref{corollary:twist}, it is equivalent to consider the twisted groups.
The twisted version of weights can be constructed in the same way as in \S\ref{subsect:paraweight}.
With Remark \ref{remark:normalizedtheta} and Notation \ref{notation:field}, the arguments are similar to those of \cite[Proposition 5.3]{LZ18} (see the proof of Proposition \ref{prop:diagonweight} in the next subsection), noting that $\sigma$ acting on $s_\Gamma^*$ corresponding to $\sigma^*$ acting on $s_\Gamma$ up to conjugacy.
For the equivalence of the two kinds of the labelling of the weights, see Proposition \ref{prop:labelweights}.
\end{proof}

%%%%%%%%%%%%%%%%%%%%%%%%%%%%%%%%%%%%%%%%%%

\subsection{Actions of diagonal automorphisms on weights}\label{subsect:diagonweight}
In this subsection, we consider the action of the diagonal automorphism on weights.
As before, we transfer to the twisted groups.
Then the diagonal automorphism is provided up to inner automorphisms by $\tau^{tw}$ as in Corollary \ref{corollary:twist}.

Let $\tau_{\Gamma,\gamma}^{tw}=\tau_{m_\Gamma,\alpha_\Gamma,\gamma}^{tw}$ be as (\ref{equation:twisteddiag}).
Since $[\tau_{\Gamma,\gamma}^{tw},R_{\Gamma,\gamma}^{tw}C_{\Gamma,\gamma}^{tw}]=1$, $\tau_{\Gamma,\gamma}^{tw}$ stabilizes $\theta_{\Gamma,\gamma}^{tw}$ and the set $\dz(N_{\Gamma,\gamma}^{tw}/R_{\Gamma,\gamma}^{tw}\mid\theta_{\Gamma,\gamma}^{tw})$.
We now consider the actions of $\tau_{\Gamma,\gamma}^{tw}$ on the set $\dz(N_{\Gamma,\gamma}^{tw}/R_{\Gamma,\gamma}^{tw}\mid\theta_{\Gamma,\gamma}^{tw})$.
The result for the case $\Gamma\neq X+1$ is the same as the field automorphisms.

\begin{lemma}\label{lemma:diagonextensiontheta}
Assume $\Gamma=X-1$ or $\Gamma\in\cF_1\cup\cF_2$.
The $\beta_\Gamma e_\Gamma$ extensions of $\theta_{\Gamma,\gamma}^{tw}$ to $C_{\Gamma,\gamma}^{tw}(V_{\Gamma,\gamma}^{tw})_{\theta_{\Gamma,\gamma}^{tw}}$ is stabilizes under $\tau_{\Gamma,\gamma}^{tw}$.
\end{lemma}
\begin{proof}
Recall that $C_{\Gamma,\gamma}^{tw}=C_\Gamma\otimes I_\gamma$ and $\theta_{\Gamma,\gamma}^{tw}=\theta_\Gamma^{tw}\otimes I_\gamma$, where $\theta_\Gamma^{tw}=\pm R_{T_\Gamma^{tw}}^{C_\Gamma^{tw}}\phi_\Gamma$.
When $\Gamma=X-1$, $\theta_\Gamma^{tw}$ is the trivial character.
By \cite[(5B)]{FS89} or direct calculation, $[\tau_{\Gamma,\gamma}^{tw},(V_{\Gamma,\gamma}^{tw})_{\theta_{\Gamma,\gamma}^{tw}}] \leqslant Z(C_{\Gamma,\gamma}^{tw})$.
Note that the irreducible linear character on $Z(C_{\Gamma,\gamma}^{tw})$ induced by $\theta_{\Gamma,\gamma}^{tw}$ is just the restriction of $\phi_\Gamma\otimes I_\gamma$.
By \cite[(6A)]{FS89}, $(\phi_\Gamma\otimes I_\gamma) ([\tau_{\Gamma,\gamma}^{tw},(V_{\Gamma,\gamma}^{tw})_{\theta_{\Gamma,\gamma}^{tw}}]) =1$.
Thus $\tau_{\Gamma,\gamma}^{tw}$ acts trivially on $C_{\Gamma,\gamma}^{tw}(V_{\Gamma,\gamma}^{tw})_{\theta_{\Gamma,\gamma}^{tw}}/\Ker\theta_{\Gamma,\gamma}^{tw}$ and the assertion follows.
\end{proof}

\begin{lemma}\label{lemma:diagondz}
Assume $\Gamma=X-1$ or $\Gamma\in\cF_1\cup\cF_2$.
Then $\tau_{\Gamma,\gamma}^{tw}$ acts trivially on the set $\dz(N_{\Gamma,\gamma,\bc}^{tw}/R_{\Gamma,\gamma,\bc}^{tw}\mid\theta_{\Gamma,\gamma,\bc}^{tw})$.
\end{lemma}
\begin{proof}
Similar as Lemma \ref{lemma:fieldondz} from Lemma \ref{lemma:extensiontheta} and Lemma \ref{lemma:diagonextensiontheta}.
\end{proof}

Now, we consider the actions of $\tau_{X+1,\gamma}^{tw}$ on the set $\dz(N_{X+1,\gamma}^{tw}/R_{X+1,\gamma}^{tw}\mid\theta_{X+1,\gamma}^{tw})$.
Note that $m_{X+1}=1$, $\alpha_{X+1}=0$, $e_{X+1}=e$.
$C_{X+1}^{tw}=T_{X+1}^{tw}\cong C_{q^e-\varepsilon}$ and $\theta_{X+1}^{tw}=\phi_{X+1}$ is the unique linear character of $T_{X+1}^{tw}$ of order $2$.
Since $C_{X+1,\gamma}^{tw}\cong C_{X+1}^{tw}\otimes I_\gamma$ and $\theta_{X+1,\gamma}^{tw}=\theta_{X+1}^{tw}\otimes I_\gamma$, $(V_{X+1,\gamma})_{\theta_{X+1,\gamma}^{tw}}=V_{X+1,\gamma}$ and $C_{X+1,\gamma}^{tw}V_{X+1,\gamma}/C_{X+1,\gamma}^{tw}$ is cyclic of order $2e$.

\begin{lemma}\label{lemma:diagonextensiontheta0}
The $2e$ extensions of $\theta_{X+1,\gamma}^{tw}$ to $C_{X+1,\gamma}^{tw}V_{X+1,\gamma}^{tw}$ can be partitioned into $e$ pairs and $\tau_{X+1,\gamma}^{tw}$ transposes each pair.
\end{lemma}
\begin{proof}
In this proof, we set $T=C_{X+1,\gamma}^{tw}\cong C_{q^e-\varepsilon}$, $\theta=\theta_{X+1,\gamma}^{tw}$, $v=v_{X+1,\gamma}^{(\varepsilon)}$, $M=M(e\ell^\gamma)$, $V=V_{X+1,\gamma}^{tw}=\langle v,M\rangle$, $\tau=\tau_{X+1,\gamma}^{tw}$.
We also set $T_0=\Ker\theta$, then $\bar{T}=T/T_0\cong C_2$ and $TV/T_0=\bar{T}\langle\bar{v},\bar{M}\rangle$ is an abelian group.
Recall that $[\tau,T]=1$.

(1) Assume $\ell$ is linear, then $\varepsilon=1$, $T\cong C_{q^e-1}$, $e$ is odd and $(q^e-1)_2=(q-1)_2$.
In this case, $o(\bar{v})=e$.
From (\ref{equation:twisteddiag}), $\tau=I_{e\ell^\gamma}\otimes\diag\{1,\xi_0\}$ with $\langle\xi_0\rangle=\F_q^\times$.
Thus  $[\tau,v]=1$ and $[\tau,M]=I_{e\ell^\gamma}\otimes\begin{bmatrix} \xi_0^{-1} & 0\\ 0 & \xi_0 \end{bmatrix}=:t\in T$.
Since $(q^e-1)_2=(q-1)_2$, $\langle\bar{t}\rangle=\bar{T}$.
Note that $M^2=-I_{2e\ell^\gamma}\in T$.

(1.1) Assume $4\mid(q-1)$.
Thus $M^2\in T_0$ and $TV/T_0=\bar{T}\times\langle\bar{v}\rangle\times\langle\bar{M}\rangle \cong C_2\times C_e\times C_2$.
Then the $2e$ extensions of $\theta$ are of the form $\ttheta(i,j)=\theta\times\mu_e^i\times\mu_2^j$, where $\langle\mu_e\rangle=\Irr(\langle\bar{v}\rangle)$, $\langle\mu_2\rangle=\Irr(\langle\bar{M}\rangle)$, $i\in\Z/e\Z$ and $j\in\Z/2\Z$.
Then by direct calculation, $\ttheta(i,j)^\tau=\ttheta(i,j+1)$.

(1.2) Assume $4\nmid(q-1)$.
Thus $\langle\bar{M}^2\rangle=\bar{T}$ and $TV/T_0=\langle\bar{v}\rangle\times\langle\bar{M}\rangle \cong C_e\times C_4$.
Fix any extension $\htheta$ of $\theta$ to $\langle\bar{M}\rangle$.
Then the $2e$ extensions of $\theta$ are of the form $\ttheta(i,j)=\mu_e^i\times\htheta^j$, where $\mu_e,i$ are as before and $j=\pm1$.
Thus, by direct calculation, $\ttheta(i,j)^\tau=\ttheta(i,-j)$.

(2) Assume $\ell$ is unitary, then $\varepsilon=-1$ and $T\cong C_{q^e+1}$.
From (\ref{equation:twisteddiag}), $$\tau=I_{\ell^\gamma}\otimes\diag\{\xi,\xi^q,\cdots,\xi^{q^{e-1}};\xi_0\xi^{-1},\xi_0\xi^{-q},\cdots,\xi_0\xi^{-q^{e-1}}\},$$
where $o(\xi)=(q-1)(q^{e}+1)$ and $\xi_0=\xi^{q^e+1}$.
In this case, $V=\langle v\rangle$ and $[\tau,v]$ is a generator of $T$.
Note that $v^{2e}=-I_{2e\ell^\gamma}$.

(2.1) Assume $4\mid(q^e+1)$.
Thus $v^{2e}\in T_0$ and $TV/T_0=\bar{T}\times\langle\bar{v}\rangle \cong C_2\times C_{2e}$.
Then the $2e$ extensions of $\theta$ are of the form $\ttheta(i)=\theta\times\mu_{2e}^i$, where $\langle\mu_{2e}\rangle=\Irr(\langle\bar{v}\rangle)$, $i\in\Z/2e\Z$.
Then by direct calculation, $\ttheta(i)^\tau=\ttheta(i+e)$.

(2.2) Assume $4\nmid(q^e+1)$.
Thus $\bar{T}\leqslant\langle\bar{v}\rangle$ and $TV/T_0=\langle\bar{v}\rangle\cong C_{4e}$.
Fix a generator $\htheta$ of $\Irr(\langle\bar{v}\rangle)$, then $\htheta$ is an extension of $\theta$ to $\langle\bar{v}\rangle$.
Then the $2e$ extensions of $\theta$ are of the form $\ttheta(j)=\htheta^j$, where $j\in\{\pm1,\pm3,\cdots,\pm(2e-1)\}$.
Thus, by direct calculation, $\ttheta(j)^\tau=\ttheta(-j)$.
\end{proof}

\begin{lemma}\label{lemma:diagondz0}
The characters of $\dz(N_{X+1,\gamma,\bc}^{tw}/R_{X+1,\gamma,\bc}^{tw}\mid\theta_{X+1,\gamma,\bc}^{tw})$ can be partitioned into pairs and $\tau_{X+1,\gamma}^{tw}$ transposes each pair.
Consequently, the $2e\ell^\delta$ characters in $\sC_{X+1,\delta}^{tw}$ can be relabelled as $\psi_{X+1,\delta,i,j}^{tw},\psi_{X+1,\delta,i,j'}^{tw}$ such that $(\psi_{X+1,\delta,i,j}^{tw})^{\tau_{X+1,\gamma}^{tw}}=\psi_{X+1,\delta,i,j'}^{tw}$.
\end{lemma}
\begin{proof}
Similar as Lemma \ref{lemma:fieldondz} from Lemma \ref{lemma:extensiontheta} and Lemma \ref{lemma:diagonextensiontheta0}.
\end{proof}

Let $(s,\kappa,i,K)^{G^*}\in i\cW(G)$; see p.\pageref{iW(G)}.
Define $K'_{X+1}: \bigcup_\delta\sC_{\Gamma,\delta} \to \{ \ell\textrm{-cores} \}$ by $K'_{X+1}(\psi_{X+1,\delta,i,j}^{tw})=K_{X+1}(\psi_{X+1,\delta,i,j'}^{tw})$ and $K'_{X+1}(\psi_{X+1,\delta,i,j'}^{tw})=K_{X+1}(\psi_{X+1,\delta,i,j}^{tw})$.
Set $K'=K'_{X+1}\prod_{\Gamma\neq X+1}K_\Gamma$.
We list the characters of $\sC_{X+1,\delta}^{tw}$ in the following way:
list characters $\psi_{X+1,\delta,i,j}^{tw}$ first; then list characters $\psi_{X+1,\delta,i,j'}^{tw}$ in the corresponding order.
Let $(s,\kappa,i,Q)$ be the label corresponding to the label $(s,\kappa,i,K)$;
see Proposition \ref{prop:labelweights} and \cite[(1A)]{AF90}.
For an $e$-quotient $\lambda^{(e)}$ of some Lusztig symbol $\lambda$, the associated ordered quotients are denoted as $(\lambda^{(e)},0),(\lambda^{(e)},1)$, we define $(\lambda^{(e)},i)'=(\lambda^{(e)},i+1)$.
Set $Q'=Q'_{X+1}\prod_{\Gamma\neq X+1}Q_\Gamma$.
Then $K'$ corresponds to $Q'$.

\begin{proposition}\label{prop:diagonweight}
Let $(R,\varphi)$ be a weight of $G$ with the label $(s,\kappa,i,K)^{G^*}$ (or the corresponding label $(s,\kappa,i,Q)^{G^*}$),
then $(R,\varphi)^\tau$ has label $(s,\kappa,i+1,K')^{G^*}$ (or the corresponding label $(s,\kappa,i+1,Q')^{G^*}$).
\end{proposition}

\begin{proof}
By Corollary \ref{corollary:twist}, it suffices to consider the twisted groups.
The proof is similar to that of \cite[Proposition 5.3]{LZ18}, but for convenience, we repeat the arguments using the twisted groups.
The construction of $\varphi^{tw}$ is the same as that of $\varphi$ before Proposition \ref{prop:labelweights}, with all constructions replaced by the twisted versions.

The diagonal automorphism in twisted groups are represented by $\tau^{tw}=\tau_0^{tw}\times\prod_{\Gamma,\delta,i}(\tau_{\Gamma,\delta,i}^{tw})^{t_{\Gamma,\delta,i}}$, where $\tau_0^{tw}$ induces the diagonal automorphism on $\I_0(V_0)$ and $\tau_{\Gamma,\delta,i}^{tw}=\tau_{m_\Gamma,\alpha_\Gamma,\gamma}^{tw}\otimes I_{\bc}$ with $\delta=\gamma+|\bc|$ and $\tau_{m,\alpha,\gamma}^{tw}$ as (\ref{equation:twisteddiag}).
Recall that $[\tau^{tw},R^{tw}C^{tw}]=1$ and $\tau^{tw}$ normalizes $N^{tw}$ and each component of $\tau^{tw}$ fixes each corresponding component of $N^{tw}$.
Since $\varphi_0^{tw}=\theta_0^{tw}=\chi_{s_0,\kappa,i}$ is an irreducible character of $\I_0(V_0)$, $(\varphi_0^{tw})^{\tau_0^{tw}}=\chi_{s_0,\kappa,i+1}$ by Proposition \ref{prop:diagonIrr}.

Set $\tau_+^{tw}=\prod_{\Gamma,\delta,i}(\tau_{\Gamma,\delta,i}^{tw})^{t_{\Gamma,\delta,i}}$.
Then $(\varphi_+^{tw})^{\tau_+^{tw}}=\Ind_{\hN_+^{tw}(\theta_+^{tw})}^{N_+^{tw}}(\psi_+^{tw})^{\tau_+^{tw}}$
and $(\psi_+^{tw})^{\tau_+^{tw}}=\prod\limits_{\Gamma,\delta,i}(\psi_{\Gamma,\delta,i}^{tw})^{(\tau_{\Gamma,\delta,i}^{tw})^{t_{\Gamma,\delta,i}}}$,
where $(\psi_{\Gamma,\delta,i}^{tw})^{(\tau_{\Gamma,\delta,i}^{tw})^{t_{\Gamma,\delta,i}}}$ is the following character
$$
\Ind_{N_{\Gamma,\delta,i}^{tw}\wr\prod_j\fS(t_{\Gamma,\delta,i,j})}^{N_{\Gamma,\delta,i}^{tw}\wr\fS(t_{\Gamma,\delta,i})}
\left(\overline{\prod_j(\psi_{\Gamma,\delta,i,j}^{tw})^{t_{\Gamma,\delta,i,j}}}\right)^{(\tau_{\Gamma,\delta,i}^{tw})^{t_{\Gamma,\delta,i,j}}}
\cdot\prod_j\phi_{\kappa_{\Gamma,\delta,i,j}}.
$$
Here, we note that $(\tau_{\Gamma,\delta,i}^{tw})^{t_{\Gamma,\delta,i}}$ acts trivially on $\fS(t_{\Gamma,\delta,i})$ and $\fS(t_{\Gamma,\delta,i,j})$.
Since we choose the extension of $\prod_j(\psi_{\Gamma,\delta,i,j}^{tw})^{t_{\Gamma,\delta,i,j}}$ to be the canonical one as in the proof of \cite[Proposition~2.3.1]{Bon99b}, we have
$$\left(\overline{\prod_j(\psi_{\Gamma,\delta,i,j}^{tw})^{t_{\Gamma,\delta,i,j}}}\right)^{(\tau_{\Gamma,\delta,i}^{tw})^{t_{\Gamma,\delta,i,j}}}
=\overline{\prod_j((\psi_{\Gamma,\delta,i,j}^{tw})^{\tau_{\Gamma,\delta,i}^{tw}})^{t_{\Gamma,\delta,i,j}}}.$$
Then the assertion follows from Lemma \ref{lemma:diagondz}, Lemma \ref{lemma:diagondz0} and the definition of the notations before the proposition.
\end{proof}

%%%%%%%%%%%%%%%%%%%%%%%%%%%%%%%%%%%%%%%%%%%%%%%%%%%%%%%%%%%
\section{Equivariance of the bijections}\label{sect:equivariance}
Finally, we can state and prove the main theorem as follows.
Recall that when $\lambda_{X+1}$ is degenerate, the irreducible Brauer characters $\varphi_{s,\lambda,0}=\varphi_{s,\lambda,1}$ which is denoted as $\varphi_{s,\lambda}$, while when $\kappa_{X+1}$ is degenerate, the weights $w_{s,\kappa,0,Q}=w_{s,\kappa,Q}$ which is denoted as $w_{s,\kappa,Q}$.
\begin{theorem}\label{theorem:equivariance}
Assume the sub-matrix of the decomposition matrix of $\Sp_{2n}(q)$ with respect to the basic set $\cE(\Sp_{2n}(q),\ell')$ is unitriangular.
The labels for irreducible Brauer characters and weights can be chosen such that the bijections established in Proposition \ref{prop:bijection} are equivariant.
\end{theorem}
\begin{proof}
The assertion for field automorphisms follows from Proposition \ref{prop:actiononBrauer} and Proposition \ref{prop:fieldonweight}.
Now, consider the diagonal automorphism; let $\tau$ generate $\tG$ modulo $G$.
Let $\phi$ be an irreducible Brauer character in a block $B$ of $G$.

(1) Assume $B=B_{s,\kappa,i}$ with $\kappa_{X+1}$ non degenerate and $\phi=\varphi_{s,\lambda,i}$.
Then the weight $(R,\varphi)$ corresponding to $\varphi_{s,\lambda,i}$ has label $(s,\kappa,i,Q(\lambda,i))$ as in Proposition \ref{prop:bijection}.
By Proposition \ref{prop:actiononBrauer}, $\varphi_{s,\lambda,i}^\tau=\varphi_{s,\lambda,i+1}\in\IBr(B_{s,\kappa,i+1})$.
By Proposition \ref{prop:diagonweight}, $(R,\varphi)^\tau$ has label $(s,\kappa,i+1,Q(\lambda,i)')$.
By the definition of $Q(\lambda,i)'$, $(Q(\lambda,i)')_\Gamma=Q(\lambda,i)_\Gamma$ for $\Gamma\neq X+1$, $(Q(\lambda,i)')_{X+1}=(Q(\lambda,i)_{X+1})'$.
Since $Q(\lambda,i)_{X+1}=(\lambda_{X+1}^{(e)},k)$ for some $k$ as in Proposition \ref{prop:bijection}, we have $(Q(\lambda,i)')_{X+1}=(\lambda_{X+1}^{(e)},k+1)$.
Since $\lambda_{X+1}=(\kappa_{X+1},i)*(\lambda_{X+1}^{(e)},k)=(\kappa_{X+1},i+1)*(\lambda_{X+1}^{(e)},k+1)$, $\varphi_{s,\lambda,i}^\tau$ corresponds to $(R,\varphi)^\tau$ by Proposition \ref{prop:bijection} and the assertion in this case follows.

(2) Assume $B=B_{s,\kappa}$ with $\kappa_{X+1}$ degenerate and $\phi=\varphi_{s,\lambda,j}$.
Then the weight $(R,\varphi)$ corresponding to $\varphi_{s,\lambda,j}$ has label $(s,\kappa,Q(\lambda,j))$ as in Proposition \ref{prop:bijection}.
Then $\varphi_{s,\lambda,j}^\tau=\varphi_{s,\lambda,j+1}$ and $(R,\varphi)^\tau$ has label $(s,\kappa,Q(\lambda,j)')$.
Again $(Q(\lambda,j)')_\Gamma=Q(\lambda,j)_\Gamma$ for $\Gamma\neq X+1$.
If $\lambda_{X+1}$ is degenerate, then $\tau$ fixes $\varphi_{s,\lambda}:=\varphi_{s,\lambda,0}=\varphi_{s,\lambda,1}$.
In this case, $\lambda_{X+1}^{(e)}$ is degenerate, thus $(Q(\lambda,j)')_{X+1}=Q(\lambda,j)_{X+1}$ and $Q(\lambda,j)'=Q(\lambda,j)$.
So, $\tau$ also fixes the weight $(R,\varphi)^G$.
Finally, assume $\lambda_{X+1}$ is non degenerate, and thus so is $\lambda_{X+1}^{(e)}$.
Then $(Q(\lambda,j)')_{X+1}=(Q(\lambda,j)_{X+1})'=(\lambda_{X+1}^{(e)},j+1)$.
Thus $\varphi_{s,\lambda,j}^\tau$ corresponds to $(R,\varphi)^\tau$ by Proposition \ref{prop:bijection} and the proof completes.
\end{proof}

\begin{remark}
We have made many choices of notations for the parameter $i\in\Z/2\Z$ for irreducible Brauer characters and weights.
This can be done compatibly since there are only two element in $\Z/2\Z$ and the diagonal automorphism and field automorphisms commute modulo inner automorphisms.
When we consider, for example, finite simple groups $\PSL_n(\varepsilon{q})$, such parameter $i$ runs through $i\in\Z/(n,q-\varepsilon)\Z$.
Then it is not so easy to choose the parameter $i$ compactibly and we may have to find some ``canonical'' ones to be labelled by $0\in\Z/(n,q-\varepsilon)\Z$.
\end{remark}

%%%%%%%%%%%%%%%%%%%%%%%%%%%%%%%%%%%%%%%%%%%%%%%%%%%%%%%%%%%%%%%%%%%%%%%%%%%%%%%%%%%%

\end{document}